\documentclass[12pt]{amsart}
\usepackage[colorlinks=true,pagebackref,hyperindex,citecolor=NavyBlue,linkcolor=Mahogany]{hyperref}
\usepackage[dvipsnames]{xcolor}
\usepackage{verbatim}
\usepackage{amsmath}
\usepackage{amsfonts}
\usepackage{amssymb}
\usepackage{color}
\usepackage[all]{xy}  
\usepackage{enumerate}
\usepackage[top=1in, bottom=1in, left=0.9in, right=0.9in]{geometry}
\usepackage{mathrsfs}

\usepackage{stmaryrd}
\usepackage{verbatim}
\usepackage{tikz}
\usetikzlibrary{shapes.geometric}
\usetikzlibrary{calc,shadows}
\usetikzlibrary{decorations.markings}

\theoremstyle{definition}
\newtheorem{theorem}{Theorem}[section]
\newtheorem{theoremx}{Theorem}

\numberwithin{equation}{section}

\newtheorem{corollary}[theorem]{Corollary}
\newtheorem{lemma}[theorem]{Lemma}
\newtheorem{proposition}[theorem]{Proposition}

\newtheorem{notation}[theorem]{Notation}

\newtheorem*{claim*}{Claim}
\newcommand{\DHS}{\Delta{\rm HS}}

\theoremstyle{definition}
\newtheorem{definition}[theorem]{Definition}

\newtheorem{example}[theorem]{Example}

\newtheorem{remark}[theorem]{Remark}

\newtheoremstyle{TheoremNum}
        {8pt}{8pt}              
        {\upshape}                      
        {}                              
        {\bfseries}                     
        {.}                             
        {.5em}                             
        {\thmname{#1}\thmnote{ \bfseries #3}}
  \theoremstyle{TheoremNum}

\newcommand{\m}{\mathfrak{m}}
\newcommand{\n}{\mathfrak{n}}

\newcommand{\ZZ}{\mathbb{Z}}
\newcommand{\QQ}{\mathbb{Q}}

\newcommand{\Hom}{\operatorname{Hom}}

\newcommand{\Ext}{\operatorname{Ext}}

\newcommand{\Ass}{\operatorname{Ass}}	
	
\newcommand{\depth}{\operatorname{depth}}

\newcommand{\Ht}{\operatorname{ht}}

\newcommand{\MM}{\mathcal{M}}
\newcommand{\NNN}{\mathcal{N}}
\newcommand{\HF}{\operatorname{HF}}
\newcommand{\HS}{\operatorname{HS}}

\newcommand{\ds}{\displaystyle}

\newcommand{\p}{\mathfrak{p}}

\newcommand{\ov}[1]{\overline{#1}}
\newcommand{\ann}{\operatorname{ann}}

\newcommand{\Q}{\mathbb{Q}}

\newcommand{\ps}[1]{\llbracket {#1} \rrbracket}

\newcommand{\PP}{\mathbb{P}}
\renewcommand{\leq}{\leqslant}
\renewcommand{\geq}{\geqslant}
\newcommand{\gin}{\operatorname{gin}}

\newcommand{\sat}{\operatorname{{sat}}}
\newcommand{\kk}{\Bbbk}

\newcommand{\soc}{\operatorname{soc}}
\newcommand{\Edepth}{\operatorname{E-depth}}
\newcommand{\In}{\operatorname{in}}
\newcommand{\rev}[1]{\operatorname{{\rm rev}_{#1}}}

\newcommand{\ale}[1]{{\color{red} \sf $\star$ Alessandro: [#1]}}

\newcommand{\giulio}[1]{{\color{blue} \sf $\star$ [#1]}}

\title[Decomposition of local cohomology tables and E-depth]{Decomposition of local cohomology tables of modules with large E-depth}
\author{Giulio Caviglia}
\address{Department of Mathematics, Purdue University, 150 N. University Street, West Lafayette, IN 47907-2067, USA}
\email{gcavigli@purdue.edu}
\author{Alessandro De Stefani}
\address{Dipartimento di Matematica, Universit{\`a} di Genova, Via Dodecaneso 35, 16146 Genova, Italy}
\email{destefani@dima.unige.it}

\subjclass[2010]{13D45, Secondary: 13P10, 13D07}
\keywords{Local cohomology tables, general initial modules, revlex-orders, sequentially Cohen-Macaulay modules}

\dedicatory{Dedicated to Professor Bernd Ulrich on the occasion of his 65th birthday}
\begin{document}

\begin{abstract}
We introduce the notion of $\Edepth$ of graded modules over polynomial rings to measure the depth of certain $\Ext$ modules. First, we characterize graded modules over polynomial rings with (sufficiently) large $\Edepth$ as those modules whose (sufficiently) partial general initial submodules preserve the Hilbert function of local cohomology modules supported at the irrelevant maximal ideal, extending a result of Herzog and Sbarra on sequentially Cohen-Macaulay modules. Second, we describe the cone of local cohomology tables of modules with sufficiently high $\Edepth$, building on previous work of the second author and Smirnov. Finally, we obtain a non-Artinian version of a socle-lemma proved by Kustin and Ulrich.
\end{abstract} 

\maketitle

\section{Introduction}
Let $S=\kk[x_1,\ldots,x_n]$ be a standard graded polynomial ring over an infinite field $\kk$. A finitely generated $\ZZ$-graded $S$-module $M$ is called sequentially Cohen-Macaulay if, for each integer $i$, the module $\Ext^i_S(M,S)$ is either zero, or Cohen-Macaulay of maximal possible dimension $n-i$. Sequentially Cohen-Macaulay modules were introduced by Stanley \cite{Stanley} from a different point of view (see Definition \ref{Defn SCM}), and later reinterpreted by Peskine as above (for instance, see \cite[Theorem 1.4]{HerzogSbarra}).

This definition suggests to consider modules for which each $\Ext^i_S(M,S)$, if not Cohen-Macaulay of maximal dimension, at least has ``sufficiently large'' depth. To better quantify this, we introduce a numerical invariant of a graded module, which we call $\Edepth$ (see Definition \ref{Definition Edepth}). Sequentially Cohen-Macaulay $S$-modules can be characterized as those which have maximal $\Edepth$, equal to $n$ (see Proposition \ref{Proposition SCM and Edepth}). On the other hand, modules with large $\Edepth$ still satisfy desirable properties. For instance, if $M$ is a module of positive depth and positive $\Edepth$, and $\ell \in S$ is a sufficiently general linear form, then the Hilbert function of the modules $H^i_\m(M)$ can be read from that of $H^{i-1}_\m(M/\ell M)$ for all $i > 0$. Here, $H^i_\m(-)$ denotes the $i$-th graded local cohomology functor, with support in the irrelevant maximal ideal $\m$ of $S$. 

Section \ref{Section Edepth} is devoted to study how the $\Edepth$ behaves under some basic operations (see Proposition \ref{proposition Edepth modulo filter regular element}), and to provide a key example of modules with positive $\Edepth$, which is crucially used in the following sections (see Example \ref{key example}). 

The starting point of Section \ref{Section gin} is an important characterization of sequentially Cohen-Macaulay modules in terms of generic initial modules, due to Herzog and Sbarra \cite{HerzogSbarra}, that we now recall. Let $\HF(-)$ denote the Hilbert function of a $\ZZ$-graded $S$-module. Let $M$ be a finitely generated $\ZZ$-graded $S$-module, that we write as a quotient of a graded free module $F$ by a graded submodule $U$. Then $M \cong F/U$ is sequentially Cohen-Macaulay if and only if $\HF(H^i_\m(F/U)) = \HF(H^i_\m(F/\gin_{\rm revlex}(U)))$.

In order to extend this, for any given integer $t \in \{0,\ldots,n\}$ we introduce a weight-order, denoted by $\rev{t}$, and we consider general initial modules $\gin_{\rev{t}}$. 

We characterize modules with sufficiently large $\Edepth$:

\begin{theoremx}[see Theorem \ref{theorem gin}] \label{THMX A}
Let $S=\kk[x_1,\ldots,x_n]$, with the standard grading, and $M$ be a finitely generated graded $S$-module. Write $M=F/U$, where $F$ is a graded free $S$-module, and $U$ is a graded submodule of $F$. For a given integer $t \geq 0$, we have that $\Edepth(M) \geq t$ if and only if $\HF(H^i_\m(F/U)) = \HF(H^i_\m(F/\gin_{{\rm rev}_t}(U)))$ for all $i \in \ZZ$.
\end{theoremx}
Keeping in mind that $M$ is sequentially Cohen-Macaulay if and only if $\Edepth(M)=n$, and that an initial ideal with respect to $\rev{n}$ coincides with the initial ideal with respect to the usual revlex order, Theorem \ref{THMX A} can be viewed as an extension of \cite[Theorem 3.1]{HerzogSbarra}.

In Section \ref{Section Decomposition}, we focus on studying the cone generated by local cohomology tables of $S$-modules $M$ with sufficiently large $\Edepth$. 

Let $M$ be a finitely generated $\ZZ$-graded $S$-module. We let $[H^\bullet_\m(M)] \in {\rm Mat}_{n+1,\ZZ}(\QQ)$ be the matrix whose $(i+1,j)$-th entry records $\dim_\kk(H^i_\m(M)_j)$, and we consider the cone 
\[
\QQ_{\geq 0} \cdot \{[H^\bullet_\m(M)] \mid M \text{ is a finitely generated } \ZZ\text{-graded } S\text{-module}\}.
\] 
The study of this object was initiated in \cite{DSS} by Smirnov and the second author. Its interest is motivated by the well-known Boij-S{\"o}derberg theory for the cone of Betti diagrams \cite{BoijSoderberg}. Eisenbud and Schreyer proved the conjectures for the cone of Betti diagrams of Cohen-Macaulay modules by exhibiting a subtle duality with the cone of cohomology tables of vector bundles on projective space \cite{ES}. Later, Boij and S{\"o}derberg extended the techniques employed in \cite{ES} to all coherent sheaves, obtaining a description of the full cone of Betti diagrams of finitely generated graded $S$-modules \cite{BoijSoderberg2}. Motivated by the original Boij-S{\"o}derberg theory, and given that local cohomology and sheaf cohomology are very much connected, Daniel Erman asked whether one could describe the cone of local cohomology tables of finitely generated graded $S$-modules. More specifically, this means whether one can identify the extremal rays of such cone, and the equations of its supporting hyperplanes.

The main results of \cite{DSS} contain a complete description of the extremal rays of the cone of local cohomology tables of modules of dimension at most $2$, as well as the equations of the supporting hyperplanes. We improve this result by determing the extremal rays of the cone spanned by modules with sufficienty large $\Edepth$.

\begin{theoremx}[see Theorem \ref{theorem decomposition LC Edepth}] \label{THMX B}
Let $S=\kk[x_1,\ldots,x_n]$, with the standard grading, and $M$ be a $\ZZ$-graded $S$-module with $\Edepth(M) \geq n-2$. For $i=0,\ldots,n$ let $S_i = \kk[x_1,\ldots,x_i]$, and let $J=(x_1,x_2)S$. We have a decomposition
\[
\ds [H^\bullet_\m(M)] = \sum_{i=0}^{n} \sum_{j \in \ZZ} r_{i,j} [H^\bullet_\m(S_i(-j))]  + \sum_{m>0} \sum_{j \in \ZZ}  r'_{m,j} [H^\bullet_\m(J^m(-j))], 
\]
where $r_{i,j} \in \ZZ_{\geq 0}$, $r'_{m,j} \in \QQ_{\geq 0}$, and all but finitely of them are equal to zero. Moreover, the  set 
\[
\ds \Lambda= \{[H^\bullet_\m(\kk[x_1,\ldots,x_s](-j)], [H^\bullet_\m(J^m(-j))] \mid 0\leq s \leq n, j \in \ZZ, m > 0\}
\]
is minimal, and it describes the extremal rays of the cone spanned by local cohomology tables of modules of $\Edepth$ at least $n-2$.
\end{theoremx}

If $M$ is an $S$-module satisfying $\Edepth(M)\geq \dim(M)-2$, then one can still apply Theorem \ref{THMX B} (see Remark \ref{Remark NN}). In particular, since modules $M$ of dimension at most two automatically satisfy $\Edepth(M) \geq \dim(M)-2$, Theorem \ref{THMX B} is indeed an extension of \cite[Theorem 4.6]{DSS}.

Using Theorem \ref{THMX B} and a description of the facets of the cone of local cohomology tables in dimension two \cite[Theorem 6.2]{DSS}, we provide equations for the supporting hyperplanes of the cone of local cohomology tables of modules $M$ with $\Edepth(M) \geq n-2$ (see Theorem \ref{facet thm}). This description becomes particularly manageable in the case of sequentially Cohen-Macaulay modules: let $\mathbb{M}$ be the $\QQ$-vector space of $(n+1) \times \ZZ$ matrices with finite support. Given the local cohomology table $[H^\bullet_\m(M)]$ of a finitely generated $\ZZ$-graded $S$-module, we can produce a new table $\Delta[H^\bullet_\m(M)]$ which belongs to $\mathbb{M}$ (see Section \ref{Section Decomposition}, or \cite[Section 6]{DSS} for more details about this construction). Consider the cone 
\[
\mathcal C^{\rm seq} = \QQ_{\geq 0} \{\Delta[H^\bullet_\m(M)] \mid M \text{ is a sequentially Cohen-Macaulay } \ZZ\text{-graded } S\text{-module}\}.
\]
\begin{theoremx}[see Proposition \ref{hyperplanes SCM} and Theorem \ref{facet thm}]
Let $A=(a_{i,j}) \in \mathbb{M}$. Then $A \in \mathcal C^{\rm seq}$ if and only if $a_{i,j} \geq 0$ for all integers $i$ and $j$.
\end{theoremx}

Finally, in Section \ref{Section soclelemma} we extend a ``socle-lemma'' due to Kustin and Ulrich to the non-Artinian case. The original version states that, if $I \subseteq J$ are two $\m$-primary homogeneous ideals, and $\HF(\soc(S/I)) \leq \HF(\soc(S/J))$, then $I=J$.
To extend this result to arbitrary dimension, we need to assume that our modules have sufficiently large $\Edepth$, and the Hilbert functions of the socles of certain local cohomology modules satisfy an analogous inequality. For simplicity, here we only state our result in the sequentially Cohen-Macaulay case:

\begin{theoremx}[see Theorem \ref{soclelemma} and Corollary \ref{soclelemma sequentially}] \label{THMX C} Let $S=\kk[x_1,\ldots,x_n]$, and $F$ be a graded free $S$-module. Let $A \subseteq B$ be graded submodules of $F$ such that $F/A$ and $F/B$ are sequentially Cohen-Macaulay. If $\HF(\soc(H^i_\m(F/A))) \leq \HF(\soc(H^i_\m(F/B)))$ for all $i \in \ZZ$, then $A=B$.
\end{theoremx}

\subsection*{Acknowledgments} We thank the anonymous referees for pointing out some inaccuracies contained in a previous version of this article, and for several very useful comments.

\section{E-depth: definitions and basic properties} \label{Section Edepth}
Let $S=\kk[x_1,\ldots,x_n]$, where $\kk$ is a field and each variable is  given degree equal to one. We will also assume that $\kk$ is infinite, since reducing to this case via a faithfully flat extension does not affect our considerations. Given a $\ZZ$-graded $S$-module $M= \bigoplus_{i \in \ZZ} M_i$, and $j \in \ZZ$, we denote by $M(j)$ its shift by $j$, that is, the $\ZZ$-graded $S$-module whose $i$-th graded component is $M_{i+j}$.

Throughout, $\m$ will always denote the maximal homogeneous ideal of $S$, and $M$ will denote a finitely generated $\ZZ$-graded $S$-module. For convenience, we let $\depth(0) = + \infty$. Given a finitely generated $\ZZ$-graded $S$-module $M$, we denote by $H^i_\m(M)$ the $i$-th graded local cohomology module of $M$, with support in $\m$. By definition, this is the $i$-th cohomology of $\check{\rm C}^\bullet \otimes_S M$, where $\check{\rm C}^\bullet$ is the $\check{{\rm C}}$ech complex on $x_1,\ldots,x_n$. 

We start by recalling the notion of sequentially Cohen-Macaulay module.

\begin{definition} \label{Defn SCM} An $S$-module $M$ is said to be {\it sequentially Cohen-Macaulay} if there exists a filtration
\[
0 = M_0 \subseteq M_1 \subseteq \ldots \subseteq M_r = M
\]
such that each quotient $M_{i+1}/M_i$ is Cohen-Macaulay with $\dim(M_{i+1}/M_i)>\dim(M_i/M_{i-1})$ for all $i=1,\ldots,r-1$.
\end{definition}

Sequentially Cohen-Macaulay modules were introduced by Stanley \cite{Stanley}. An equivalent formulation, due to Peskine, is the following: $M$ is sequentially Cohen-Macaulay if and only if, for every $i \in \ZZ$, the module $\Ext^{i}_S(M,S)$ is either zero, or Cohen-Macaulay of dimension $n-i$. 

\begin{example} 
Cohen-Macaulay modules are sequentially Cohen-Macaulay. One dimensional modules are also sequentially Cohen-Macaulay, since either $M$ is Cohen-Macaulay, or the filtration $0 \subseteq H^0_\m(M) \subseteq M$ has Cohen-Macaulay subquotients, and $\dim(M/H^0_\m(M)) = 1> 0 = \dim(H^0_\m(M))$.
\end{example}

We observe that, if $M$ is a sequentially Cohen-Macaulay $S$-module and $\ell$ is a linear non-zero divisor on $M$, as well as on $\Ext^i_S(M,S)$ and  $\Ext^{i+1}_S(M,S)$, then 
\[
\Ext^{i+1}_S(M/\ell M,S) \cong \Ext^i_S(M(-1),S)/\ell \Ext^i_S(M,S).
\] 
In fact, we will see that $\ell$ only needs to be a non-zero divisor on $M/H^0_\m(M)$ and on the two $\Ext$ modules for this to be true, not necessarily on $M$. This simple observation often allows to reduce the dimension of a sequentially Cohen-Macaulay module, yet controlling features such as depth and regularity. In this sense, the case when $M$ is sequentially Cohen-Macaulay is the best possible, since all $\Ext$-modules have maximal depth. We introduce the notion of $\Edepth$ of a module $M$ to measure the number of times that the above procedure can be re-iterated, without altering the cohomological features of $M$.

\begin{definition} \label{Definition Edepth}
Let $S=\kk[x_1,\ldots,x_n]$, and $M$ be a finitely generated graded $S$-module. For an integer $t \in \ZZ_{\geq 0}$, we say that $M$ satisfies condition $(E_t)$ if $\depth(\Ext^i_S(M,S)) \geq \min \{t,n-i\}$ for all $i$. We define
\[
\ds \Edepth(M) = \min\bigg\{n, \ \sup\{t \in \ZZ_{\geq 0} \mid M \text{ satisfies } (E_t)\}\bigg\}.
\]
\end{definition}

\begin{remark} The definition of $\Edepth$ is here given in the standard graded setting, because it suits the level of  generality that we consider in this article. The same definitions, and completely analogous considerations, can be made for finitely generated modules over local rings.
\end{remark}

We now study some basic properties of the $\Edepth$ of a module. We start by noticing that there is no general relation between $\Edepth(M)$ and $\depth(M)$, even when $M$ is a $\kk$-algebra.
\begin{example} \label{ex finite length LC 2dim}
Let $S=\kk[x,y,z,w]$, and let $R=S/\p$, where $\p$ is the kernel of the map $\varphi:S \to \kk[s,t]$ defined as follows:
\[
\ds \varphi(x) = s^4, \ \varphi(y) = s^3t, \ \varphi(z) = st^3, \ \varphi(w) = t^4.
\]
The only two non-zero $\Ext$ modules are $\Ext^2_S(R,S)$ and $\Ext^3_S(R,S)$. It can be checked that $\Ext^3_S(R,S)$ has finite length, and therefore $\Edepth(R) = 0$ is forced. On the other hand, $\depth(R) = 1$.
\end{example}
\begin{example}
Let $S=\kk[x,y]$ and $R=S/I$, with $I=(x^2,xy)$. It is clear that, $\depth(R)=0$. On the other hand, the only two non-zero $\Ext$ modules $\Ext^1_S(R,S)$ and $\Ext^2_S(R,S)$ are both Cohen-Macaulay of dimension one and zero, respectively. So $R$ is sequentially Cohen-Macaulay, and thus $\Edepth(R)=2$.
\end{example}

We recall the definition of filter and strictly filter regular sequence.

\begin{definition}
Let $S=\kk[x_1,\ldots,x_n]$, and $M$ be a $\ZZ$-graded $S$-module. A homogeneous element $\ell \in \m$ is called a {\it filter regular element for $M$} if $\ell \notin \bigcup_{\p \in \Ass^\circ(M)} \p$, where $\Ass^\circ(M) = \Ass(M) \smallsetminus \m$. A sequence $\ell_1,\ldots,\ell_t$ is called a {\it filter regular sequence for $M$} if $\ell_i$ is filter regular for $M/(\ell_1,\ldots,\ell_{i-1})M$ for all $i$.
\end{definition}

Equivalently, $\ell$ is filter regular for $M$ if $0:_M \ell$ has finite length, and in this case one has
\[
H^0_\m(M)  = 0:_M \m^{\infty} \subseteq 0:_M \ell^\infty \subseteq H^0_\m(M),
\]
hence forcing equality everywhere. A related notion is that of strictly filter regular element.

\begin{definition}
Let $S=\kk[x_1,\ldots,x_n]$, and $M$ be a $\ZZ$-graded $S$-module. For all $i \in \ZZ$, let $X^i=\Ass^\circ(\Ext^i_S(M,S))$. A homogeneous element $\ell \in \m$ is called a {\it strictly filter regular element for $M$} if $\ell \notin \bigcup_{i \in \ZZ} \bigcup_{\p \in X^i} \p$. A sequence $\ell_1,\ldots,\ell_t$ is called a {\it strictly filter regular sequence for $M$} if $\ell_i$ is strictly filter regular for $M/(\ell_1,\ldots,\ell_{i-1})M$ for all $i$.
\end{definition}

We will simply say that $\ell_1,\ldots,\ell_t$ is a filter (resp. strictly filter) regular sequence whenever the module $M$ is clear from the context. It follows from \cite[11.3.9]{BrodmannSharp} that $\Ass^\circ(M) \subseteq \bigcup_i \bigcup_{\p \in X^i} \p$, therefore a strictly filter regular sequence is automatically a filter regular sequence.

\begin{lemma}\label{lemma positive Edepth ses} Let $S=\kk[x_1,\ldots,x_n]$, and $M$ be a finitely generated $\ZZ$-graded $S$-module. Let $N=M/H^0_\m(M)$, and $\ell$ be a strictly filter regular element for $M$ of degree $\delta>0$. We have that $\Edepth(M)>0$ if and only if the graded sequences
\[
\xymatrix{
0 \ar[r] & \Ext^{n-i}_S(M,S) \ar[r]^-{\cdot \ell} & \Ext^{n-i}_S(M(-\delta),S) \ar[r] & \Ext^{n-i+1}_S(N/\ell N,S) \ar[r] & 0
}
\]
\[
\xymatrix{
0 \ar[r] & H^{i-1}_\m(N/\ell N) \ar[r] & H^i_\m(M)(-\delta)\ar[r]^-{\cdot \ell} & H^i_\m(M) \ar[r] & 0
}
\]
induced by $0 \to N(-\delta) \stackrel{\cdot \ell}{\longrightarrow} N \to N/\ell N \to 0$ are exact for all $i>0$.
\end{lemma}
\begin{proof}
By previous observations, we have that $\ell$ is also filter regular for $M$, hence it is regular for $N$. Assume that $\Edepth(M)>0$, and consider the graded short exact sequence $0 \to N(-\delta)\stackrel{\cdot \ell}{\longrightarrow} N \to N/\ell N \to 0$. This gives a long exact sequence
\[
\xymatrix{
\cdots \ar[r] & \Ext^{n-i}_S(N,S) \ar[r]^-{\cdot \ell} & \Ext^{n-i}_S(N(-\delta),S) \ar[r] & \Ext^{n-i+1}_S(N/\ell N,S) \ar[r] & \cdots
}
\]
Observe thatb $\Ext^{n-i}_S(N,S) = 0$ for $i \leq 0$. Moreover, it follows from the short exact sequence $0 \to H^0_\m(M) \to M \to N \to 0$ that $\Ext^n_S(M,S) = \Ext^n_S(H^0_\m(M),S)$, while $\Ext^{n-i}_S(M,S) \cong \Ext^{n-i}_S(N,S)$ for all $i>0$. As $\Edepth(M)>0$, and $\ell$ is strictly filter regular, we have that $\ell$ is regular on $\Ext^{n-i}_S(M,S) \cong \Ext^{n-i}_S(N,S)$ for all $i>0$. In particular, multiplication by $\ell$ on $\Ext^i_S(M,S)$ in the long exact sequence above is injective for all $i>0$, and the long exact sequence breaks into short exact sequences. The statement for local cohomology modules follows at once from graded local duality \cite[13.4.3]{BrodmannSharp}.

Conversely, assume that the sequences above are exact. From the $\Ext$-sequence we deduce that either $\Ext^{n-i}_S(M,S)=0$, or $\ell$ is a non-zero divisor for it. In particular, we have that $\depth(\Ext^{n-i}_S(M,S))>0$ for all $i>0$. Since $\Ext^n_S(M,S)$ has finite length, and $\Ext^{n-i}_S(M,S)=0$ for $i<0$, it follows that $\Edepth(M)>0$.
\end{proof}

Our next goal is to provide a more explicit relation between $\Edepth$ and sequentially Cohen-Macaulay modules. We first need a lemma.

\begin{lemma} \label{Lemma d-1->d} Let $S=\kk[x_1,\ldots,x_n]$, with the standard grading, and $M$ be a finitely generated $\ZZ$-graded $S$-module of dimension $d$ such that $\depth(\Ext^{n-d}_S(M,S)) \geq d-1$. If $d >1$, further assume that $\depth(\Ext^{n-(d-1)}_S(M,S))>0$. Then $\depth(\Ext^{n-d}_S(M,S))=d$.
\end{lemma}
\begin{proof} We proceed by induction on $d \geq 0$. If $d=0$ there is nothing to show. If $d=1$, then let $\ell$ be a strictly filter regular element for $M$ of degree $\delta>0$, and $N=M/H^0_\m(M)$. From the short exact sequence $0 \to N(-\delta) \stackrel{\cdot 
\ell}{\longrightarrow} N \to N/\ell N \to 0$ we obtain an exact sequence
\[
\xymatrixcolsep{6mm}
\xymatrix{
\Ext^{n-1}_S(N/\ell N,S) \ar[r] & \Ext^{n-1}_S(N,S) \ar[r]^-{\cdot \ell} & \Ext^{n-1}_S(N(-\delta),S) \ar[r] &  \Ext^{n}_S(N/\ell N,S).
}
\]
However, since $\dim(N/\ell N)=0$, we have $\Ext^{n-1}_S(N/\ell N,S)=0$, that is, $\ell$ is a non-zero divisor on $\Ext^{n-1}_S(N,S)$. Since $\Ext^{n-1}_S(M,S) \cong \Ext^{n-1}_S(N,S)$, we have that $\depth(\Ext^{n-1}_S(M,S)) = 1$. If $d>1$, then let $\ell$ be a strictly filter regular element for $M$, and consider the same short exact sequence as above, which gives an exact sequence
\[
\xymatrixcolsep{5mm}
\xymatrix{
0 \ar[r] & \Ext^{n-d}_S(N,S) \ar[r]^-{\cdot \ell} & \Ext^{n-d}_S(N(-\delta),S) \ar[r] & \Ext^{n-(d-1)}_S(N/\ell N,S) \ar[r] & \ann_{\Ext^{n-(d-1)}_S(N,S)}(\ell),
}
\]
where the zero on the left follows again from the fact that $\Ext^{n-d}_S(N/\ell N,S)=0$, since $\dim(N/\ell N) <d$. As above, we also have $\Ext^{n-d}_S(M,S) \cong \Ext^{n-d}_S(N,S)$ and $\Ext^{n-(d-1)}_S(M,S) \cong \Ext^{n-(d-1)}_S(N,S)$. Since $\depth(\Ext^{n-(d-1)}_S(M,S)) = \depth(\Ext^{n-(d-1)}_S(N,S))>0$, and because $\ell$ is a strictly filter regular element, we have that $\ann_{\Ext^{n-(d-1)}_S(N,S)}(\ell)=0$. In particular, we obtain that
\[
\depth(\Ext^{n-(d-1)}_S(N/\ell N,S)) = \depth(\Ext^{n-d}_S(M,S)) - 1 \geq d-2.
\]
Applying the inductive hypothesis to the module $N/\ell N$, which has dimension $d-1$, gives that $\depth(\Ext^{n-(d-1)}_S(N/ \ell N,S)) = d-1$, and thus $\depth(\Ext^{n-d}_S(M,S))=d$, as claimed.
\end{proof}

\begin{proposition} \label{Proposition SCM and Edepth}
Let $S=\kk[x_1,\ldots,x_n]$, with the standard grading, and $M$ be a finitely generated $\ZZ$-graded $S$-module. The following are equivalent:
\begin{enumerate}[(a)]
\item $M$ is sequentially Cohen-Macaulay.
\item $\Edepth(M)=n$.
\item $M$ satisfies condition $(E_t)$ for some $t \geq \dim(M)-1$.
\end{enumerate}
\end{proposition}
\begin{proof}
The implications (a) $\Rightarrow$ (b) $\Rightarrow$ (c) are clear from the definitions. Let $d=\dim(M)$, and assume that $M$ satisfies $(E_t)$ for some $t \geq d-1$. Let $i$ be such that $\Ext^i_S(M,S) \ne 0$. Since $\Ht(\ann_S(M))=n-d$, we must have $i \geq n-d$. For $i > n-d$ we have that $t \geq d-1 \geq n-i$, and thus
\[
n-i = \min\{t,n-i\} \leq \depth(\Ext^i_S(M,S)) \leq \dim(\Ext^i_S(M,S)) \leq n-i.
\]
In particular, $\Ext^i_S(M,S)$ is Cohen-Macaulay of dimension $n-i$. If $i=n-d$, by assumption we have that $\depth(\Ext^{n-d}_S(M,S)) \geq \min\{t,d\} \geq d-1$. 
Since $\dim(\Ext^{n-d}_S(M,S)) = d$, and because $\depth(\Ext^{n-(d-1)}_S(M,S)) = d-1 >0$ when $d>1$, we conclude by Lemma \ref{Lemma d-1->d} that $\Ext^{n-d}_S(M,S)$ is Cohen-Macaulay of dimension $d$, 
and therefore $M$ is sequentially Cohen-Macaulay.
\end{proof}

\begin{proposition} \label{proposition Edepth modulo filter regular element} Let $S=\kk[x_1,\ldots,x_n]$, with the standard grading, and $M$, $M'$ be two finitely generated graded $S$-modules. We have:
\begin{enumerate}
\item \label{prop direct sum} $\Edepth(M \oplus M') = \min \{\Edepth(M),\Edepth(M') \}$.
\item \label{prop H^0} $\Edepth(M) = \Edepth(M/H^0_\m(M))$.
\item \label{prop modulo filter} Let $N=M/H^0_\m(M)$. If $\Edepth(M)>0$ and $\ell$ is a homogeneous strictly filter regular element, then either $\Edepth(N/\ell N) = \Edepth(M)= n$, or $\Edepth(N/\ell N) = \Edepth(M)-1$.
\end{enumerate}
\end{proposition}
\begin{proof}
The proof of (\ref{prop direct sum}) follows immediately from the definitions. Let $N=M/H^0_\m(M)$. For $i \ne n$, we have that $\Ext^i_S(M,S) \cong \Ext^i_S(N,S)$. Since $\Ext^n_S(N,S)=0$, while $\Ext^n_S(M,S)$ has finite length, it is clear that $M$ satisfies condition $(E_t)$ for some $t$ if and only if $N$ does, and part (\ref{prop H^0}) follows. We now prove (\ref{prop modulo filter}). Let $N = M/H^0_\m(M)$. If $\Edepth(M)=n$, then $M$ is sequentially Cohen-Macaulay by Proposition \ref{Proposition SCM and Edepth}. By (\ref{prop H^0}) it follows that $N$ is sequentially Cohen-Macaulay, and so is $N/\ell N$ by \cite[Corollary 1.9]{HerzogSbarra}. In particular, as a module over $S$, 
we have that $\Edepth(N/\ell N)=n$, again by Proposition \ref{Proposition SCM and Edepth}. Now assume that $\Edepth(M) = t < n$. By part (\ref{prop H^0}) we have that $\Edepth(N) =t$, which is positive by assumption. By Lemma \ref{lemma positive Edepth ses}, if we let $\delta$ be the degree of $\ell$, we have graded short exact sequences
\[
\xymatrix{
0 \ar[r] & \Ext^i_S(M,S) \ar[r]^-{\cdot \ell} & \Ext^i_S(M(-\delta),S) \ar[r] & \Ext^{i+1}_S(N/\ell N,S) \ar[r] & 0
}
\]
for all $i < n$. Let $i$ be such that $\Ext^i_S(M,S) \ne 0$. The short exact sequences above show that 
\[
\depth(\Ext^{i+1}_S(N/\ell N,S)) = \depth(\Ext^{i}_S(M,S)) - 1 \geq \min\{t-1,n-(i+1)\}
\]
As this holds for all $i+1 \leq n$, we have that $\Edepth(N/\ell N) \geq t-1$. On the other hand, since $\Edepth(M) = t < n$, there must exist $i$ such that $\depth(\Ext^i_S(M,S)) =t < n-i$. Then $\depth(\Ext^{i+1}_S(N/\ell N,S)) = t-1 < n-(i+1)$, which shows that $\Edepth(N/\ell N)=t-1$.
\end{proof}

\begin{remark} Observe that if $\Edepth(M)$ is not assumed to be positive in Proposition \ref{proposition Edepth modulo filter regular element} (\ref{prop modulo filter}), then $\Edepth(N/\ell N)$ can even increase. Indeed, Example \ref{ex finite length LC 2dim} exhibits an integral $\kk$-algebra $R$ such that $\Edepth(R)=0$, but $\Edepth(R/\ell R) = 4$ for any non-zero linear form $\ell$, since $R/\ell R$ is one-dimensional, hence sequentially Cohen-Macaulay over $S=\kk[x,y,z,w]$.
\end{remark}

We conclude the section by providing examples of classes of modules with a given $\Edepth$. The relevance of the following construction will become clearer in the upcoming sections.

Let $t \geq 0$, and $R=A[y_1,\ldots,y_t]$ be a polynomial ring over a $\ZZ$-graded ring $A$. 
We put a $\ZZ \times \ZZ^t$-grading on $R$ as follows. Let $\eta_i \in \ZZ^{t+1}$ be the vector with $1$ in position $i$ and $0$ everywhere else. We set $\deg_R(a) = \deg_A(a) \cdot \eta_1$ for all $a \in A$, and $\deg_R(y_i) = \eta_{i+1}$.

\begin{example} Let $A=\kk[x_1,x_2]$, with the standard grading, and $R=A[x_3,x_4]$. Then $R$ is $\ZZ \times \ZZ^2$-graded. For instance, we have $\deg_R(x_1) = \deg_R(x_2) = (1,0,0)$, while $\deg_R(x_3) = (0,1,0)$ and $\deg_R(x_4) = (0,0,1)$.
\end{example}

Viewing $S$ as $A[y_1,\ldots,y_t]$, where $A=\kk[x_1,\ldots,x_{n-t}]$ and $y_i=x_{n-t+i}$, we see that $S$ is a $\ZZ\times \ZZ^t$-graded ring. Observe that a non-zero polynomial of $S$ is homogeneous with respect to this grading if and only if it is a monomial in the last $t$ variables, and is homogeneous with respect to the standard grading in the first $n-t$ variables. Throughout, whenever we claim that a module is $\ZZ \times \ZZ^t$ graded for some $t \geq 0$, we mean that it is graded with respect to this grading. When $t=0$, this simply means that the module 
is $\ZZ$-graded with respect to the standard grading on $S$.

Similar considerations can be done in the subrings $S_j = \kk[x_1,\ldots,x_j]$ of $S$. That is,  if $j \geq n-t$, we can view $S_j$ as $A_j[y_1,\ldots,y_{j-(n-t)}]$, where $A_j = \kk[x_1,\ldots,x_{n-t}]$ and $y_i = x_{n-t+i}$. In this way, $S_j$ is a $\ZZ \times \ZZ^{j-(n-t)}$-graded ring. 

Observe that, if $M$ is a $\ZZ \times \ZZ^t$-graded $S$-module, with $t>0$, we have that $M/x_nM$ is still $\ZZ \times \ZZ^t$-graded, and it can be identified with a $\ZZ \times \ZZ^{t-1}$-graded module over $S_{n-1}$. 

\begin{example}[Key Example] \label{key example}
Let $S=\kk[x_1,\ldots,x_n]$, and $M$ be a finitely generated $\ZZ \times \ZZ^t$-graded $S$-module such that $x_n,\ldots,x_{n-t+1}$ is a filter regular sequence for $M$. Then $\Edepth(M) \geq t$, and $x_n,\ldots,x_{n-t+1}$ is a strictly filter regular sequence for $M$.  

In fact, we can write $M=F/U$, where $F$ is a free $S$-module, and $U$ is a $\ZZ \times \ZZ^t$-graded submodule of $F$. Consider the saturation $U^{\sat} = U:x_n^\infty = \{\alpha \in F \mid x_n^r \alpha \in U$ for some $r \gg 0\}$. Since $x_n$ is a $\ZZ \times \ZZ^t$-homogeneous element, we have that $U^{\sat}$ is $\ZZ \times \ZZ^t$-graded itself, and so is $F/U^{\sat}$. Since $x_n$ is assumed to be filter regular, we actually have $F/U^{\sat} \cong M/H^0_\m(M)$. If $F/U^{\sat}=0$, then $M$ is zero dimensional, hence sequentially Cohen-Macaulay. In particular, $\Edepth(M) = n \geq t$, and any filter regular sequence is automatically strictly filter regular. Assume that $F/U^{\sat} \ne 0$, so that $x_n$ is a non-zero divisor on $F/U^{\sat}$. Since $F/U^{\sat}$ is $\ZZ \times \ZZ^t$-graded, we can write it as 
$F/U^{\sat} = \ov F/\ov U \otimes_\kk \kk[x_n]$, where $\ov F$ is a free graded $S_{n-1}$-module, and $\ov U$ is a $\ZZ \times \ZZ^{t-1}$-graded submodule of $\ov F$ such that $\ov{F}/\ov{U}$ can be identified with $F/U^{\sat} \otimes_S S/x_n S$. In particular, for all $i < n$ we have
\[
\Ext^i_S(M,S) \cong \Ext^i_S(F/U^{\sat},S) \cong \Ext^i_{S_{n-1}}(\ov F/\ov U,S_{n-1})\otimes_\kk \kk[x_n].
\]
Hence $x_n$ is a non-zero divisor on $\Ext^i_S(M,S)$ for all $i <n$, and it is then a strictly filter regular element for $M$. Moreover, we have that $\Edepth(M) >0$. Iterating this argument $t$-times gives the desired claim. 
\end{example}

We will make a more systematic use of the methods of Example \ref{key example} in the next sections. 

\section{Partial general initial modules and E-depth}
\label{Section gin}

Given integers $0 \leq t \leq n$, we consider the following $t \times n$ matrix:
\[
\Omega_{t,n}=\begin{bmatrix}
0 & 0 & \ldots & 0 & 0 & 0&  \ldots & 0 & -1 \\
0 & 0 & \ldots & 0 & 0 & 0&  \ldots & -1 & 0 \\
\vdots & \vdots & \vdots & \vdots & \vdots & \vdots&  \vdots & \vdots & \vdots \\
0 &  0 & \ldots & 0 & 0 &-1 & \ldots & 0 & 0 \\ 
0 &  0 & \ldots & 0 & -1 &0 & \ldots & 0 & 0 \\ 
0 & 0 & \ldots & -1 & 0 &0 & \ldots & 0 & 0
\end{bmatrix} 
\]
If we let $S=\kk[x_1,\ldots,x_n]$, then $\Omega_{t,n}$ induces a ``partial revlex'' term order on $S$. Given a finitely generated $\ZZ$-graded $S$-module $M$, we can present it as $M=F/U$, where $F$ is a finitely generated $\ZZ$-graded free $S$-module, with graded basis $\{e_1,\ldots,e_r\}$. Notice that an element $f \in F$ can be written uniquely as a finite sum of monomials of $F$, that is, we can write $f=\sum_j u_j e_{i_j}$ where the elements $u_j$ are monomials in $S$ and the sum has minimal support. Then, the initial form ${\rm in}_{{\rm rev}_t}(f)$ of $f$ with respect to the grading induced by $\Omega_{t,n}$ will be the sum of elements of the form $u_je_{i_j}$ from $f$ for which $u_j$ is maximal with respect to the order induced by $\Omega_{t,n}$ on $S$. Observe that, in general, ${\rm in}_{{\rm rev}_t}(f)$ may not be of the form $f'e_i$ for some $i=1,\ldots,r$. In other words, it may not live in one single free summand of $F$. And even if it is of that form, the coefficient $f'$ may not be a monomial of $S$. 

Given that the order ${\rm in}_{{\rm rev}_t}$ can be extended to $F$, it makes sense to consider the initial submodule ${\rm in}_{{\rm rev}_t}(U)$ of $U$ in $F$.
\begin{remark} \label{Remark gin} One can check that the one defined is a 
partial reverse lexicographic order (see \cite[15.7]{Eisenbud} for more details). In particular, we have 
\begin{enumerate}[(i)]
\item ${\rm in}_{{\rm rev}_t}(U:_Fx_n^s) = {\rm in}_{{\rm rev}_t}(U):_F x_n^s$ for all $s > 0$.
\item ${\rm in}_{{\rm rev}_t}(U+x_n F) = {\rm in}_{{\rm rev}_t}(U)+x_n F$.
\end{enumerate}
More generally, one could take the partial orders induced by the following $(t+1) \times n$ matrix:
\[
\Omega_{t,n}'=\begin{bmatrix}
0 & 0 & \ldots & 0 & 0 & 0&  \ldots & 0 & -1 \\
0 & 0 & \ldots & 0 & 0 & 0&  \ldots & -1 & 0 \\
\vdots & \vdots & \vdots & \vdots & \vdots & \vdots&  \vdots & \vdots & \vdots \\
0 &  0 & \ldots & 0 & 0 &-1 & \ldots & 0 & 0 \\ 
0 &  0 & \ldots & 0 & -1 &0 & \ldots & 0 & 0 \\ 
0 & 0 & \ldots & -1 & 0 &0 & \ldots & 0 & 0 \\
1 & 1 & \ldots & 1 & 1 &1 & \ldots & 1 & 1
\end{bmatrix},
\]
which takes also the degree of a monomial into account, or 
the one induced by $\Omega'_{t,n}$, and that successively defines that $ue_i > ve_j$ if $i<j$. In both these cases, properties (i) and (ii) listed above are still satisfied.  
Similarly, we would like to point out that one can also take into account the degrees of a graded basis of $F$. However, in order to define a revlex order (according to \cite[15.7]{Eisenbud}) satisfying properties (i) and (ii) above, such degrees should be considered only after all rows of $\Omega_{t,n}$ have been evaluated. 
\end{remark}

For the rest of this section we assume that $\kk$ is infinite. The goal is to define a ``partial general initial submodule'' of a given submodule $U$ of a free $S$-module $F$. 
\begin{definition} \label{Defn gin}
Let $S=\kk[x_1,\ldots,x_n]$, $F$ be a $\ZZ$-graded free $S$-module, and $U$ be a graded $S$-submodule of $F$. We say that the 
{\it partial general initial submodule} of $U$ satisfies a given property (P) if there exists a non-empty Zariski open set $\mathscr L$ of $t$-uples of linear forms such that for every point $\ell=(\ell_{n-t+1},\ldots,\ell_n) \in \mathscr L$ the module $F/{\rm in}_{{\rm rev}_t}({\bf g}_\ell(U))$ satisfies property (P), where ${\bf g}_\ell$ is the change of coordinates sending $\ell_i \mapsto x_i$ and that fixes the other variables.
\end{definition}

For instance, we will consider properties (P) such as having a specific Hilbert function, or a specific value for regularity, Betti numbers, etc. 

In fact, it is easy to see that such invariants and the corresponding non-empty Zariski open set where the property is constantly true or constantly false can be computed in the following way: let $\widetilde{\kk} = \kk(\alpha_{ij} \mid n-t+1 \leq n \leq t, 1 \leq j \leq n)$ be a purely transcendental field extension of $\kk$, and let $\widetilde{\ell}_i = \sum_{j=1}^n \alpha_{ij} x_j$. Consider the change of coordinates ${\bf \widetilde{g}}_\ell$ sending $\widetilde{\ell}_i \mapsto x_i$ and that fixes the other variables, and compute ${\rm in}_{{\rm rev}_t}({\bf \widetilde{g}}_\ell(U))$ and any of the invariants mentioned above. The algorithm for such calculations is based on repeated Gr{\"o}bner bases computations. Collect all non-zero coefficients in $\widetilde{\kk}$ which appear in the calculations. 
Observe that they are finitely many rational functions in $\kk[\alpha_{ij} \mid n-t+1 \leq i \leq n, 1 \leq j \leq n]$. We set $\mathscr{L}$ to be the Zariski open set of points where such functions are defined, and 
do not vanish. Since $\kk$ is infinite, the intersection is not empty.

By abusing notation, we will call any such submodule a general partial initial submodule of $U$, and denote it by $\gin_{{\rm rev}_t}(U)$. Thanks to the discussion above, we will therefore consider features such as {\it the} Hilbert function, {\it the} Betti numbers, and {\it the} 
Hilbert function of local cohomology modules of $\gin_{{\rm rev}_t}(U)$.

Let $\mathscr L$ be a Zariski open set consisting of $t$-uples of linear forms, that we can view as a Zariski open set in a projective space $\PP = \PP^{(n-1) \times t}$. To each point $\ell \in \mathscr L$ is associated a linear change of coordinates ${\bf g}_\ell$ defined as above. Vice versa, to each ${\bf g}_\ell$ we can associate a point $\ell \in \PP$. By abusing notation, we will henceforth refer to a Zariski open set of transformations of the form ${\bf g}_\ell$ to mean the above scenario.

Now consider the closed subspace $\PP^{\rm UP}$ of $\PP$ consisting of ``upper triangular'' $t$-uples of linear forms, that is, elements of the form $(\ell_{n-t+1},\ldots,\ell_n)$ where $\ell_i$ is a linear form supported on the variables $x_1,\ldots,x_i$. Observe that, for $\ell \in \mathscr L \cap \PP^{\rm UP}$, the corresponding change of coordinates ${\bf g}_\ell$ can be represented by an upper triangular matrix.
\begin{remark} \label{Remark UP}
In order to test whether a property (P) of a general initial submodule holds, it is sufficient to produce a non-empty Zariski open set in $\PP^{\rm UP}$ where (P) holds. In fact, let $\mathscr L$ be a Zariski open set of $t$-uples of linear forms. Associated to $\ell \in \mathscr L$ we have a change of coordinates ${\bf g}_\ell$ as in Definition \ref{Defn gin}, which can be represented as a matrix of the form
\begin{equation}
\label{form matrices}
{\bf g}_\ell = \begin{bmatrix} I_{n-t} & \star \\
0 & \star
\end{bmatrix},
\end{equation}
where $I_{n-t}$ is the identity matrix of size $n-t$, and $\begin{bmatrix} \star \\ \star \end{bmatrix}$ has size $n \times t$. By possibly shrinking the open set $\mathscr L$, we can factor such a matrix in the product of a lower triangular matrix with all entries equal to one on the main diagonal, and an upper triangular matrix: ${\bf g}_\ell = {\bf g}^{\rm LOW}_{\ell} {\bf g}^{{\rm UP}}_\ell$. Moreover, 
\begin{equation}
\label{Eq Aldo}
{\rm in}_{\rev{t}}({\bf g}_\ell(U)) = {\rm in}_{\rev{t}}( {\bf g}^{\rm LOW}_{\ell} {\bf g}^{{\rm UP}}_\ell(U)) = {\rm in}_{\rev{t}}(  {\bf g}^{{\rm UP}}_\ell(U)),
\end{equation}
where the last equality follows from standard properties of revlex-type orders. Thus, starting from a non-empty Zariski open set $\mathscr L \subseteq \PP$ where property (P) holds, one can produce a non-empty Zariski open set $\mathscr L^{\rm UP} \subseteq \PP^{\rm UP}$ where (P) still holds. Vice versa, assume that we are given a non-empty Zariski open set $\mathscr L^{\rm UP}$ inside $\PP^{\rm UP}$, so that the change of coordinates ${\bf g}_\ell$ corresponding to points in $\mathscr L^{\rm UP}$ are upper triangular. By acting on the set of such transformations with the following group of $n \times n$ matrices
\[
\left\{A=\begin{bmatrix} I_{n-t} & 0 \\ 0 & \star\end{bmatrix} \ \bigg| \  A=\{a_{ij}\} \text{ is lower triangular, and } a_{i,i}=1 \text{ for all } i=1,\ldots,n \right\},
\]
one obtains a non-empty Zariski open set of matrices ${\bf g}_\ell$ of the form (\ref{form matrices}) on which property (P) still holds by (\ref{Eq Aldo}). In other words, this gives a non-empty Zariski open subset $\mathscr L \subseteq \PP$ where (P) still holds.
\end{remark}

The following is an extension of \cite[Proposition 2.14]{Green} to our setting, which will be used in the proof of the main result of this section. Even if the argument is similar, it is more technical. Thus, we provide a proof for sake of completeness.
\begin{lemma} \label{Lemma gin Green} Let $S = \kk[x_1,\ldots,x_n]$ with the standard grading, $N$ be a non-negative integer and $t$ be a positive integer. Let $F$ be a free $S$-module, and $U$ be a graded submodule of $F$. For any sufficiently general linear form $h = \sum_{i=1}^n \alpha_i x_i$, we can identify $((U:h^N)+h F)/hF$ with a submodule $V_h$ of a free $S_{n-1} = \kk[x_1,\ldots,x_{n-1}]$-module $\ov F$ by setting $x_n = -\alpha_n^{-1}(\sum_{i=1}^{n-1} \alpha_i x_i)$. Consider a property (P). There exists a non-empty Zariski open set of linear forms $\mathscr H$ such that, for all $h \in \mathscr H$, the module $((\gin_{\rev{t}}(U):x_n^N)+x_nF)/x_n F$ satisfies (P) if and only if $\ov{\gin}_{\rev{t-1}}(V_{h})$ satisfies (P). Here, $\ov{\gin}_{\rev{t-1}}$ denotes a general partial initial submodule computed in $S_{n-1}$.
\end{lemma}
\begin{proof}
First of all, observe that to compute ${\rm in}_{\rev{t}}$ we can first compute the initial submodule ${\rm in}_{\rev{1}}$ with respect to the first row of the matrix $\Omega_{t,n}$ introduced above, and then compute the initial submodule with respect to the remaining $t-1$ rows, that we denote by ${\rm in}_{\ov{\rev{t-1}}}$. 

By Remark \ref{Remark UP}, we can reduce to considering upper triangular changes of coordinates. In particular, we can find a non-empty Zariski open set $\mathscr L^{\rm UP} \subseteq \left(\PP^{(n-1)\times t} \right)^{\rm UP}$ such that the linear change of coordinates ${\bf g}_{\ell} = {\bf g}_{\ell}^{\rm UP}$ introduced in Definition \ref{Defn gin} is upper triangular and 
the given property (P) holds for $((\gin_{\rev{t}}(U):x_n^N)+x_n F)/x_n F$ if and only if it holds for $(({\rm in}_{\rev{t}}({\bf g}_\ell^{\rm UP}(U)):x_n^N)+x_n F)/x_n F$ for all $\ell \in  \mathscr L^{\rm UP}$.

For $\ell  = (\lambda_{n-t+1},\ldots,\lambda_{n-1},\lambda_n) \in \mathscr L^{\rm UP}$, set $\lambda = (\lambda_{n-t+1},\ldots,\lambda_{n-1})$. We can factor ${\bf g}_{\ell}^{\rm UP} $ as the composition ${\bf \ov{g}}_{\lambda}^{\rm UP} \circ {\bf k}_{\lambda_n}$, defined as follows: ${\bf k}_{\lambda_n}$ is the change of coordinates that fixes $x_i$ for $i \ne n$, and sends $\lambda_n \mapsto x_n$, while ${\bf \ov{g}}_{\lambda}^{\rm UP}$ is the change of coordinates  such that $x_i \mapsto x_i$ for $1 \leq i \leq n-t$ and $i=n$, and such that $\lambda_i \mapsto x_i$ for $n-t+1 \leq i \leq n-1$.

Let $U' = {\bf k}_{\lambda_n}(U)$. Proceeding in a similar manner as in \cite[Section 6]{Green}, where Green constructs partial elimination ideals for the lex order, we can write $U'$ as a disjoint union of sets
\[
U'= U'_{[0]} \sqcup U'_{[1]} x_n \sqcup U'_{[2]} x_n^2 \sqcup \cdots,
\]
where $U'_{[i]}x_n^i$ consists of the elements of $U'$ that are divisible by $x_n^i$, and not by $x_n^{i+1}$. It can easily be checked that $U'_{[i]} \subseteq U'_{[j]}$ if $i \leq j$. Given a polynomial $f=f(x_1,\ldots,x_n) \in S$, we set $\ov f = f(x_1,\ldots,x_{n-1},0) \in S_{n-1}$. Now, if $(f_1,\ldots,f_r) \in F = S^{\oplus r}$, we define $\ov{(f_1,\ldots,f_r)}=(\ov{f_1},\ldots,\ov{f_r}) \in \ov F$, where $\ov F$ is a free $S_{n-1}$-module, which can be identified with $F/x_n F$. Since $U'_{[i]}$ is a subset of $F$, it makes sense to define $\ov{U'_{[i]}} = \{\ov u \mid u \in U_{[i]}\}$ and $\ov{{\bf \ov g}_\lambda^{\rm UP}(U'_{[i]})}=\{\ov u \mid u \in {\bf \ov{g}}_{\lambda}^{\rm UP}(U'_{[i]})\}$. Since the change of coordinates ${\bf \ov{g}}_{\lambda}^{\rm UP}$ fixes $x_n$, and each other linear form $\lambda_i$ involved in such transformation does not have $x_n$ in its support, one can check that 
\begin{align*}
{\rm in}_{\rev{1}}({\bf \ov{g}}_{\lambda}^{\rm UP}(U')) & = \ov{{\bf \ov g}_\lambda^{\rm UP}(U'_{[0]})} \sqcup \ov{{\bf \ov g}_\lambda^{\rm UP}(U'_{[1]})} x_n \sqcup \ov{{\bf \ov g}_\lambda^{\rm UP}(U'_{[2]})} x_n^2 \sqcup \cdots \\
& = \ov{\bf{g}}_{\lambda}^{\rm UP}\left(\ov{U'_{[0]}} \sqcup \ov{U'_{[1]}} x_n \sqcup \ov{U'_{[2]}} x_n^2 \sqcup \cdots \right) = {\bf \ov{g}}_{\lambda}^{\rm UP}({\rm in}_{\rev{1}}(U')).
\end{align*}
Therefore
\begin{align*}
({\rm in}_{\rev{t}}({\bf g}_{\ell}^{\rm UP}(U)):x_n^N) + x_nF  & = ({\rm in}_{\rev{t}}({\bf \ov{g}}_{\lambda}^{\rm UP}(U')) : x_n^N) + x_n F \\
& = ({\rm in}_{\ov{\rev{t-1}}}({\bf \ov{g}}_{\lambda}^{\rm UP}({\rm in}_{\rev{1}}(U'))) : x_n^N) + x_n F \\
& = ({\rm in}_{\ov{\rev{t-1}}}({\bf \ov{g}}_{\lambda}^{\rm UP}(\ov{U'_{[0]}} \sqcup \ov{U'_{[1]}}x_n \sqcup \ov{U'_{[2]}} x_n^2 \sqcup \cdots)) : x_n^N) + x_n F \\ 
& = \left({\rm in}_{\ov{\rev{t-1}}}\left(\ov{{\bf \ov{g}}_{\lambda}^{\rm UP}(U'_{[0]})} \sqcup \ov{{\bf \ov{g}}_{\lambda}^{\rm UP}(U'_{[1]})}x_n \sqcup \cdots\right) : x_n^N\right) + x_n F.
\end{align*}
Because of how ${\rm in}_{\ov{\rev{t-1}}}$ is defined, we have that ${\rm in}_{\ov{\rev{t-1}}}\left(\ov{{\bf \ov{g}}_{\lambda}^{\rm UP}(U'_{[i]})}x_n^i\right) = {\rm in}_{\ov{\rev{t-1}}}\left(\ov{{\bf \ov{g}}_{\lambda}^{\rm UP}(U'_{[i]})}\right)x_n^i$. Therefore the last formula is equal to
\[
\left(\left({\rm in}_{\ov{\rev{t-1}}}\left(\ov{{\bf \ov{g}}_{\lambda}^{\rm UP}(U'_{[0]})}\right) \sqcup {\rm in}_{\ov{\rev{t-1}}}\left(\ov{{\bf \ov{g}}_{\lambda}^{\rm UP}(U'_{[1]})}\right)x_n \sqcup \cdots \right) : x_n^N\right) + x_n F.
\]
One can check that, as a set, ${\rm in}_{\ov{\rev{t-1}}}\left(\ov{{\bf \ov{g}}_{\lambda}^{\rm UP}(U'_{[i]})}x_n^i\right):x_n^N$ equals ${\rm in}_{\ov{\rev{t-1}}}\left(\ov{{\bf \ov{g}}_{\lambda}^{\rm UP}(U'_{[i]})}x_n^{i-N}\right)$ if $i \geq N$, and it equals ${\rm in}_{\ov{\rev{t-1}}}\left(\ov{{\bf \ov{g}}_{\lambda}^{\rm UP}(U'_{[i]})}\right)$ if $i<N$. Since $U'_{[i]} \subseteq U'_{[j]}$ if $i \leq j$, the above expression is equal to
\begin{align*}
\left({\rm in}_{\ov{\rev{t-1}}}\left(\ov{{\bf \ov{g}}_{\lambda}^{\rm UP}(U'_{[N]})}\right) \sqcup {\rm in}_{\ov{\rev{t-1}}}\left(\ov{{\bf \ov{g}}_{\lambda}^{\rm UP}(U'_{[N+1]})}\right)x_n \sqcup \cdots \right) + x_n F & = {\rm in}_{\ov{\rev{t-1}}}\left(\ov{{\bf \ov{g}}_{\lambda}^{\rm UP}(U'_{[N]})}\right) + x_n F \\
& = {\rm in}_{\ov{\rev{t-1}}}\left({\bf \ov{g}}_{\lambda}^{\rm UP}(\ov{U'_{[N]}})\right) + x_n F.
\end{align*}

Let $\pi: \left(\PP^{(n-1)\times t}\right)^{\rm UP} \to \PP^{n-1}$ be the map which sends $\ell=(\lambda_{n-t+1},\ldots,\lambda_n)$ to $\lambda_n$, and let $\mathscr H= \pi(\mathscr L^{\rm UP})$, which is a non-empty Zariski open set. Moreover, let $\mathscr L_h^{\rm UP} = \{\lambda=(\lambda_{n-t+1},\ldots,\lambda_{n-1}) \mid (\lambda_{n-t+1},\ldots,\lambda_{n-1},h) \in \mathscr L^{\rm UP}\}$. We have shown that, in order to decide whether $((\gin_{\rev{t}}(U):x_n^N) + x_n F)/x_n F$ satisfies the given property (P), one can just check whether for $h \in \mathscr H$ and $\lambda = (\lambda_{n-t+1},\ldots,\lambda_{n-1}) \in \mathscr L_h^{\rm UP}$, after the change of coordinates ${\bf k}_{\lambda_n}$ this is true for the submodule ${\rm in}_{\ov{\rev{t-1}}}({\bf \ov{g}}_{\lambda}^{\rm UP}(\ov{U'_{[N]}}))$ of the $S_{n-1}$-module $\ov F$.

For any $h \in \mathscr H$, we have that ${\bf k}_{h}((U:h^N)+h F) = (U':x_n^N) + x_n F = U'_{[N]} + x_n F$. Thus, after the change of coordinates ${\bf k}_h$, we may identify $V_{h}$ with $\ov{U'_{[N]}}$. Note that, for $\lambda \in \mathscr L_h^{\rm UP}$, the transformation ${\bf \ov{g}}_\lambda^{\rm UP}$ can be viewed as an upper triangular change of coordinates in $S_{n-1}$.

By possibly shrinking the open set $\mathscr L_h^{\rm UP}$, we may assume that the module  $\ov{\gin}_{\rev{t-1}}(V_h)$ satisfies (P) if and only if ${\rm in}_{\ov{\rev{t-1}}}({\bf \ov{g}}_\lambda^{\rm UP}(\ov{U'_{[N]}}))$ does for every $\lambda \in \mathscr L_h^{\rm UP}$. 
Finally, observe that a partial initial submodule ${\rm in}_{\ov{\rev{t-1}}}$ computed over $\ov{F}$ is the same as a partial initial submodule with respect to the matrix $\Omega_{t-1,n-1}$. Putting all these facts together, by Definition \ref{Defn gin} we can then say that for $h \in \mathscr H$ the module $\ov{\gin}_{\rev{t-1}}(V_h)$ satisfies (P) if and only if $((\gin_{\rev{t}}(U):x_n^N) + x_n F)/x_n F$ does, and this concludes the proof.
\end{proof}

We exhibit a first relation between $\gin_{\rev{t}}$ of a module and the notion of $\Edepth$.

\begin{proposition} \label{gin has Edepth}
Let $S=\kk[x_1,\ldots,x_n]$, and $M$ be a $\ZZ$-graded $S$-module of dimension $n$. Write $M=F/U$, where $F$ is a graded free $S$-module. If $t$ is an integer with $0 \leq t \leq n$, then $\Edepth(F/\gin_{{\rm rev}_t}(U)) \geq t$.
\end{proposition}
\begin{proof}
Observe that, by construction, $F/\gin_{{\rm rev}_t}(U)$ is $\ZZ \times \ZZ^t$-graded. In the notation of Definition \ref{Defn gin}, we have that $x_n,\ldots,x_{n-t+1}$ forms a filter regular sequence for $F/({\bf g}_\ell(U))$. It follows from \ref{Remark gin} (i) that they also form a filter regular sequence for $F/{\rm in}_{{\rm rev}_t}({\bf g}_\ell(U)) = F/\gin_{\rev{t}}(U)$. 
The claim now follows from Example \ref{key example}.
\end{proof}

Recall that, given a graded $S$-module $M$, we denote 
by $\HF(M)$ its Hilbert function, that is, the row vector whose entry in position $j$ equals $\dim_\kk(M_j)$.

The following is the main theorem of this section. As it will be pointed out later, this can be viewed as extension of the main result of Herzog and Sbarra in \cite{HerzogSbarra}. We follow closely the steps of their proof.

\begin{theorem} \label{theorem gin}
Let $S=\kk[x_1,\ldots,x_n]$, and $M$ be a finitely generated $\ZZ$-graded $S$-module, that we can write as a quotient $M=F/U$, where $F$ is a graded free $S$-module. For a given integer $0 \leq t \leq n$, we have that $\Edepth(M) \geq t$ if and only if $\HF(H^i_\m(F/U)) = \HF(H^i_\m(F/\gin_{{\rm rev}_t}(U)))$ for all $i \in \ZZ$.
\end{theorem}
\begin{proof}
After a general change of coordinates, we may assume that $V= {\rm in}_{{\rm rev}_t}(U)$ has the same properties as $\gin_{\rev{t}}(U)$. We may also assume that $x_n,\ldots,x_{n-t+1}$ forms a strictly filter regular sequence for $F/U$. 
Since strictly filter regular sequences are filter regular sequences, it follows from Remark \ref{Remark gin} (i) that $x_n,\ldots,x_{n-t+1}$ forms a filter regular sequence for $F/V$. Since $V$ is $\ZZ \times \ZZ^t$-graded, it then follows from Example \ref{key example} that  $x_n,\ldots,x_{n-t+1}$ forms a strictly filter regular sequence also for $F/V$. Furthermore, by Remark \ref{Remark gin} we have
\[
V^{\sat} = \In_{\rev t}(U):_F x_n^\infty = \In_{\rev t}(U:_F x_n^\infty) = \In_{\rev t}(U^{\sat}),
\]
and also
\begin{equation} \label{Eq revlex}
V^{\sat}+x_n F = \In_{\rev t}(U^{\sat}) + x_n F = \In_{\rev t}(U^{\sat} + x_n F).  \
\end{equation}

We first prove that if $\Edepth(M) \geq t$, then there is equality for the Hilbert functions of local cohomology modules. We proceed by induction on $t \geq 0$, the case $t=0$ being trivial (note that $\In_{\rev 0}(U)=U$).
By Proposition \ref{proposition Edepth modulo filter regular element} (\ref{prop H^0}) we have that $\Edepth(F/U^{\sat}) \geq t>0$, where $U^{\sat} = U:_F x_n^{\infty}$. It follows from Proposition \ref{proposition Edepth modulo filter regular element} (\ref{prop modulo filter}) that $\Edepth(F/(U^{\sat}+x_n F)) \geq t-1$. By induction we have that $\HF(H^i_\m(F/(U^{\sat}+x_n F))) = \HF(H^i_\m(F/(V^{\sat}+x_n F)))$ for all $i \in \ZZ$.
Since $\Edepth(F/U^{\sat}) \geq t>0$, by Lemma \ref{lemma positive Edepth ses} we have short exact sequences 
\[
\xymatrix{
0 \ar[r] & H^i_\m(F/(U^{\sat}+x_n F)) \ar[r] & H^{i+1}_\m(F/U)(-1) \ar[r]^-{\cdot x_n} & H^{i+1}_\m(F/U) \ar[r] & 0
}
\]
for all $i \geq 0$. By Proposition \ref{gin has Edepth} and Lemma \ref{lemma positive Edepth ses}, we have analogous short exact sequences for $F/V$:
\[
\xymatrix{
0 \ar[r] & H^i_\m(F/(V^{\sat}+x_n F)) \ar[r] & H^{i+1}_\m(F/V)(-1) \ar[r]^-{\cdot x_n} & H^{i+1}_\m(F/V) \ar[r] & 0.
}
\]
Let $\HS(M) = \sum_{j  \in \ZZ} \dim_\kk(M_j) z^j$ be the Hilbert series of $M$. For $i \geq 0$, we have
\begin{align*}
\HS(H^{i+1}_\m(F/U))(z-1)& = \HS(H^i_\m(F/(U^{\sat}+x_n F))) \\
& = \HS(H^i_\m(F/(V^{\sat}+x_n F))) =\HS(H^{i+1}_\m(F/V))(z-1),
\end{align*}
and thus $\HF(H^i_\m(F/U)) = \HF(H^i_\m(F/V))$ for all $i >0$. Finally, we have $\HF(F/U) = \HF(F/V)$, and also $\HF(F/U^{\sat}) = \HF(F/{\rm in}_{{\rm rev}_t}(U^{\sat})) = \HF(F/V^{\sat})$, by Remark \ref{Remark gin} (i). Therefore we conclude that
\[
\HF(H^0_\m(F/U)) = \HF(U^{\sat}/U) = \HF(V^{\sat}/V) = \HF(H^0_\m(F/V)).
\]

To prove the converse, assume that the local cohomology modules of $F/U$ and $F/V$ have the same Hilbert function. We want to prove by induction on $t \geq 0$ that $\Edepth(F/U) \geq t$. If $t=0$ there is nothing to show.  By Proposition \ref{proposition Edepth modulo filter regular element} (\ref{prop H^0}), it suffices to show that $\Edepth(F/U^{\sat}) \geq t$. 
We claim that the set $\mathcal S=\{j \in \ZZ_{\geq 0} \mid H^j_\m(F/U^{\sat}) \stackrel{\cdot x_n}{\longrightarrow} H^j_\m(F/U^{\sat})$ is not surjective$\}$ is empty. If not, let $i = \min \mathcal S$, and observe that $i>0$, since $H^0_\m(F/U^{\sat})=0$. Then we have an exact sequence
\[
\xymatrix{
0 \ar[r] & H^{i-1}_\m(F/(U^{\sat}+x_n F)) \ar[r] & H^i_\m(F/U)(-1) \ar[r]^-{\cdot x_n} & H^i_\m(F/U).
}
\]
Since the rightmost map is not surjective by choice of $i$, there exists $j \in \ZZ$ such that
\begin{align*}
\dim_\kk(H^{i-1}_\m(F/(U^{\sat}+x_n F))_j) & > \dim_\kk(H^i_\m(F/U)_{j-1})-\dim_\kk(H^i_\m(F/U)_j) \\
& = \dim_\kk(H^i_\m(F/V)_{j-1})-\dim_\kk(H^i_\m(F/V)_j) \\
 & = \dim_\kk(H^{i-1}_\m(F/(V^{\sat}+x_n F))_j).
\end{align*}
Here, we used that $\dim_\kk(H^i_\m(F/U)_j) = \dim_\kk(H^i_\m(F/V)_j)$ for all $j \in \ZZ$ by assumption, and that the sequence 
 \[
\xymatrix{
0 \ar[r] & H^{i-1}_\m(F/(V^{\sat}+x_n F)) \ar[r] & H^i_\m(F/V)(-1) \ar[r]^-{\cdot x_n} & H^i_\m(F/V) \ar[r] & 0
}
\]
is exact by Proposition \ref{gin has Edepth} and Lemma \ref{lemma positive Edepth ses}. However, by upper semi-continuity and using (\ref{Eq revlex}), we obtain
\[
\HF(H^{i-1}_\m(F/(U^{\sat}+x_n F))) \leq \HF(H^{i-1}_\m(F/{\rm in}_{{\rm rev}_t}(U^{\sat}+x_n F))) = \HF(H^{i-1}_\m(F/(V^{\sat}+x_n F))).
\]  
This contradicts the inequality obtained above, and hence the set $\mathcal S$ is empty. This implies that $\Edepth(F/U)>0$ by Lemma \ref{lemma positive Edepth ses}, and also that $\HF(H^i_\m(F/(U^{\sat}+x_n F))) = \HF(H^i_\m(F/(V^{\sat}+x_n F)))$ for all $i \in \ZZ$. 
Viewing $F/(U^{\sat}+x_n F)$ as a quotient $\ov F/\ov U$ of a finitely generated free $S_{n-1}$-module $\ov F$, one 
still has equality of Hilbert functions for the local cohomology modules of $\ov F/\ov U$ and $\ov F/\ov V$, where $\ov V={\rm in}_{{\rm rev}_{t-1}}(\ov U)$. Moreover, by Lemma \ref{Lemma gin Green}, $\ov V$ has the same properties as $\gin_{\rev{t-1}}(\ov U)$. 
By induction, we then have $\Edepth(F/(U^{\sat}+ x_n F)) \geq t-1$ as a module over $S_{n-1}$. Since $\Edepth(F/U^{\sat}) = \Edepth(F/U)>0$ by Proposition \ref{proposition Edepth modulo filter regular element} (\ref{prop H^0}), we conclude by Proposition \ref{proposition Edepth modulo filter regular element} (\ref{prop modulo filter}) that either $\Edepth(F/U) = \Edepth(F/U^{\sat}+x_nF)= n \geq t$, or $\Edepth(F/U)<n$ and $\Edepth(F/U) = \Edepth(F/(U^{\sat}+x_n F))+1 \geq t$. In both cases, the desired inequality is obtained.
\end{proof}

Observe that the proof of Theorem \ref{theorem gin} can be adapted to the general initial submodule of any revlex order which satisfies conditions (i), (ii) of Remark \ref{Remark gin}, and Lemma \ref{Lemma gin Green}. In particular, we recover \cite[Theorem 3.1]{HerzogSbarra}.
\begin{corollary} \label{gin SCM}
Let $S = \kk[x_1,\ldots,x_n]$, and $M$ be a finitely generated $\ZZ$-graded $S$-module. Write $M=F/U$, where $F$ is a free $S$-module and $U$ is a graded submodule. Then $M$ is sequentially Cohen-Macaulay if and only if $\HF(H^i_\m(F/U)) = \HF(H^i_\m(F/\gin_{{\rm revlex}}(U)))$ for all $i \in \ZZ$.
\end{corollary} 

\section{Decomposition of local cohomology tables} \label{Section Decomposition}

For convenience of the reader, we recall the grading introduced in Section \ref{Section Edepth}. Let $t \geq 0$, and $R=A[y_1,\ldots,y_t]$ be a polynomial ring over a $\ZZ$-graded ring $A$. 
We put a $\ZZ \times \ZZ^t$-grading on $R$ as follows. Let $\eta_i \in \ZZ^{t+1}$ be the vector with $1$ in position $i$ and $0$ everywhere else. We set $\deg_R(a) = \deg_A(a) \cdot \eta_1$ for all $a \in A$, and $\deg_R(y_i) = \eta_{i+1}$. 
We recall the notation $S_j$ for the subring $\kk[x_1,\ldots,x_j]$ of $S$, and we let $\m_j = (x_1,\ldots,x_j)S_j$. If $j \geq n-t \geq 0$, we recall that $S_j$ is a $\ZZ \times \ZZ^{j-(n-t)}$-graded ring. 
\begin{lemma} \label{lemma family N}
Let $S=\kk[x_1,\ldots,x_n]$, and $M$ be a $\ZZ \times \ZZ^t$-graded module. 
Assume that $x_n,x_{n-1},\ldots,x_{n-t+1}$ forms a filter regular sequence for $M$. 
There exists a family $\{\NNN_j\}_{j=n-t}^n$ of modules that satisfies the following conditions:
\begin{itemize}
\item $\NNN_n = M$.
\item Each $\NNN_j$ is either zero, or is a finitely generated $\ZZ \times \ZZ^{j-(n-t)}$-graded module over $S_j$ of Krull dimension $\dim(\NNN_j) = \dim(M)-(n-j)$. 
\item $x_{j},\ldots,x_{n-t+1}$ is a strictly filter regular sequence on $\NNN_{j}$ for every $j=n-t+1,\ldots,n$.
\item For all $j$ we have $\NNN_j/H^0_{\m_j}(\NNN_j) \cong \NNN_{j-1} \otimes_\kk \kk[x_{j}]$.
\end{itemize}
\end{lemma}
\begin{proof}
We construct such modules inductively, starting from $\mathcal N_n = M$. 
Assume that $\NNN_j$ has been constructed for some $n-t+1 \leq j \leq n$. Since $\NNN_j$ is $\ZZ \times \ZZ^{j-(n-t)}$-graded over $S_{j}$, and $x_{j}$ is filter regular over $\NNN_j$, we have that $\NNN_j/H^0_{\m_j}(\NNN_j) = \NNN_{j-1} \otimes_\kk \kk[x_j]$ for some finitely generated $\ZZ \times \ZZ^{(j-1)-(n-t)}$-graded $S_{j-1}$-module, that we call $\NNN_{j-1}$. From this, we see that $\NNN_j/H^0_{\m_j}(\NNN_j) \otimes_{S_{j}}S_{j} /x_{j}S_j$ can be identified with $\NNN_{j-1}$. 
If $\dim(\NNN_j) = 0$, then $\NNN_j/H^0_{\m_j}(\NNN_j) = 0$, and thus $\NNN_{j-1}=0$. Otherwise $\NNN_{j-1}$ has dimension $\dim(\NNN_j)-1 = \dim(M)-(n-j+1)$. Moreover, for $i<j$ we have
\[
\Ext^i_{S_j}(\NNN_j,S_j) \cong \Ext^i_{S_j}(\NNN_j/H^0_{\m_j}(\NNN_j),S_j) \cong \Ext^i_{S_{j-1}}(\NNN_{j-1},S_{j-1}) \otimes_\kk \kk[x_j].
\]
This shows that $x_j$ is a non-zero divisor on $\Ext^i_{S_j}(\NNN_j,S_j)$ for all $i<j$, and therefore it is a strictly filter regular element for 
$\NNN_j$. By construction, we see that it is also strictly filter regular for $\NNN_s$, for all $s \geq j$, and the proof is complete. 
\end{proof}

\begin{notation} \label{Notation filtration}
Let $S=\kk[x_1,\ldots,x_n]$, and $M$ be a finitely generated $\ZZ$-graded $S$-module such that $x_n,\ldots,x_{n-j+1}$ is a filter regular sequence for $M$. We consider the following chain of submodules of $M$:
\[
\mathcal{Q}_0 = 0:_M x_n^\infty \subseteq \mathcal{Q}_1 = (x_n)M:_M x_{n-1}^\infty \subseteq \ldots \subseteq \mathcal{Q}_j=(x_{n-j+1},\ldots,x_n)M:_M x_{n-j}^\infty.
\]
We set $\mathcal{M}_0=\mathcal{Q}_0$, and for $1 \leq i \leq j$ we let
\[
\MM_i = 
\frac{\mathcal{Q}_i}{\mathcal{Q}_{i-1}+x_{n-i+1}M}.
\]
\end{notation}
Observe that $(x_{n-i+1},\ldots,x_n)\MM_i=0$, hence we will view $\MM_i$ 
as a finitely generated $\ZZ$-graded $S_{n-i}$-module. Because of its definition, it is also a finite dimensional graded $\kk$-vector space. 
\begin{remark} \label{Remark N = M}
If $\{\NNN_j\}_{j=n-t}^n$ is the family of modules constructed in Lemma \ref{lemma family N} for $M$, then one can readily check that $H^0_{\m_{n-j}}(\NNN_{n-j}) \cong \MM_{j}$ for all $j$ as graded $\kk$-vector spaces.
\end{remark}

\begin{proposition} \label{proposition LC}
Let $S=\kk[x_1,\ldots,x_n]$, and $M$ be a finitely generated 
$S$-module. 
Assume that $M$ is $\ZZ \times \ZZ^t$-graded for some $t \geq 1$, and that $x_n,x_{n-1},\ldots,x_{n-t+1}$ form a filter regular sequence on $M$. Let $\{\NNN_j\}_{j=n-t}^n$ be the family of modules constructed in Lemma \ref{lemma family N}. We have graded $\kk$-vector space isomorphisms: 
\[
H^j_\m(M) \cong \begin{cases} \MM_j \otimes_\kk H^j_{\m_j}(S_j) & \text{ if } 0 \leq j \leq t-1  \\
H^{j-t}_{\m_{n-t}}(\NNN_{n-t}) \otimes_\kk H^t_{\m_t}(S_t) & \text{ if } t \leq j \leq n
\end{cases}
\]
\end{proposition}
\begin{proof}
First, observe that 
\[
\ds H^0_\m(M) = \MM_0 \cong \MM_0 \otimes_\kk \kk = \MM_0 \otimes_\kk H^0_{\m_0}(S_0),
\]
so the statement is trivially true for $j=0$. By the K{\"u}nneth formula for local cohomology, and by Lemma \ref{lemma family N}, for $j>0$ and $0 \leq s \leq \min\{j,t\}$ we have graded $\kk$-vector space isomorphisms: 
\begin{align*}
H^j_\m(M) & \cong 
H^j_\m(\NNN_{n-s}\otimes_\kk \kk[x_{n-s+1},\ldots,x_n]) \\
& \cong \bigoplus_i \left(H^{j-i}_{\m_{n-i}}(\NNN_{n-s}) \otimes_\kk H^i_{(x_{n-s+1},\ldots,x_n)}(\kk[x_{n-s+1},\ldots,x_n])\right) \\
&\cong H^{j-s}_{\m_{n-s}}(\NNN_{n-s}) \otimes_\kk H^s_{\m_s}(S_s),
\end{align*}
where we used that $\kk[x_{n-s+1},\ldots,x_n] \cong S_s$ and the only non-vanishing cohomology of $S_s$ occurs for $i=s$. In particular, for $j \geq t$ we have $H^j_\m(M) \cong H^{j-t}_{\m_{n-t}}(\NNN_{n-t}) \otimes_\kk H^t_{\m_t}(S_t)$, while for $j \leq t-1$ we have $H^j_\m(M) \cong H^0_{\m_{n-j}}(\NNN_{n-j}) \otimes_\kk H^j_{\m_j}(S_j)$, which by Remark \ref{Remark N = M} is isomorphic to $\MM_j \otimes_\kk H^j_{\m_j}(S_j)$.
\end{proof}

\begin{remark} \label{Remark modulo} In what follows, we will view $S_j = \kk[x_1,\ldots,x_j]$ as an $S$-module, identifying it with $S/(x_{j+1},\ldots,x_n)$. In this way, the two local cohomology modules $H^j_\m(S_j)$ and $H^j_{\m_j}(S_j)$ are isomorphic, and we will switch from one to the other without further justifications.
\end{remark}

\begin{remark} \label{Remark NN} Let $M$ be a finitely generated $S$-module of dimension $d$. Consider a graded Noether normalization $R$ of $S/\ann_S(M)$, and observe that $M$ is a finitely generated graded $R$-module. Since $R$ is isomorphic to a polynomial ring in $d$ many variables, and because local cohomology does not change when viewing $M$ as an $R$-module, we may view the local cohomology table of $M$ both as an $S$ and an $R$-module, with no distinctions. Moreover, we may always assume that $\kk$ is infinite without affecting considerations on the cone of local cohomology tables (see \cite[Lemma 2.2]{DSS}). Finally, if $M$ is not sequentially Cohen-Macaulay one can show that $\Edepth(M)$ is unaffected by viewing it as an $R$-module instead of an $S$-module. On the other hand, if $M$ is sequentially Cohen-Macaulay as an $S$-module, it is also sequentially Cohen-Macaulay as an $R$-module, with $\Edepth(M)=d$. 
\end{remark}

We now turn our attention to local cohomology tables. Let $S=\kk[x_1,\ldots,x_n]$, with the standard grading. Let $M$ be a finitely generated graded $\ZZ$-module. We consider the local cohomology table of $M$:
\[
[H^\bullet_\m(M)] = (h^i_j) \in {\rm Mat}_{n+1,\ZZ}(\ZZ)
\] 
where $h^i_j := \dim_\kk(H^i_\m(M)_j)$. We can then consider the cone spanned by the local cohomology tables, with non-negative rational coefficients:
\[
\QQ_{\geq 0} \cdot \{[H^\bullet_\m(M)] \mid M \text{ is a finitely generated } \ZZ\text{-graded } S\text{-module}\}.
\] 
A description of the extremal rays and supporting hyperplanes of this cone has been given in \cite{DSS} when restricting to local cohomology tables of modules of dimension at most two. 

Given a finitely generated $\ZZ$-graded module $M$, we will consider its (local cohomology) Hilbert series:
\[
\HS(H^\bullet_\m(M)) = \sum_{i=0}^n \sum_{j \in \ZZ} h^i_j u^iz^j \in \ZZ\ps{z^{\pm 1}}[u].
\]

If $N$ is a finitely generated $\ZZ$-graded $S_{n-j}$-module, then by the proof of Proposition \ref{proposition LC} we have 
\begin{equation}
\label{Eq Delta}
\HS(H^\bullet_\m(N \otimes_\kk \kk[x_{n-j+1},\ldots,x_n])) = \frac{u^j}{(z(1-z^{-1}))^j} \HS(H^\bullet_\m(N)) = \frac{u^j}{(z-1)^j} \HS(H^\bullet_\m(N)).
\end{equation}

We start by proving a decomposition theorem for $S$-modules satisfying $\Edepth(M) \geq \dim(S)-2$. By Remark \ref{Remark NN}, if $\Edepth(M) \geq \dim(M)-2$, we can still reduce to this case. 
In particular, since modules of dimension at most two 
automatically satisfy this condition, the following is an extension of \cite[Theorem 4.6]{DSS}.

\begin{theorem} \label{theorem decomposition LC Edepth}
Let $S=\kk[x_1,\ldots,x_n]$ be a standard graded polynomial ring, and $M$ be a $\ZZ$-graded $S$-module. Let $t=\Edepth(M)$, and assume that $t \geq n-2$. Let $J=(x_1,x_2)S$. 
We have a decomposition
\[
\ds [H^\bullet_\m(M)] = \sum_{i=0}^{n} \sum_{j \in \ZZ} r_{i,j} [H^\bullet_\m(S_i(-j))]  + \sum_{m>0} \sum_{j \in \ZZ}  r'_{m,j} [H^\bullet_\m(J^m(-j))], 
\]
where $r_{i,j} \in \ZZ_{\geq 0}$, $r'_{m,j} \in \QQ_{\geq 0}$, and all but finitely many of them are equal to zero. Moreover, the set 
\[
\ds \Lambda =\{[H^\bullet_\m(S_i(-j))], [H^\bullet_\m(J^m(-j))] \mid 0\leq i \leq n, j \in \ZZ, m >0\}
\] 
is minimal, that is, none of the tables from $\Lambda$ can be written as a non-negative rational linear combination of the other tables from the same set. 
\end{theorem} 
\begin{proof}
We may assume that $n \geq 2$, otherwise the result follows from \cite[Theorem 4.6]{DSS}. Since $t \geq n-2$, by Theorem \ref{theorem gin} we can replace $M$ by a partial general initial module $F/\gin_{\rev{n-2}}(U)$ which is $\ZZ \times \ZZ^{n-2}$-graded, and is such that $x_n,\ldots,x_3$ forms a filter regular sequence. Let $h^i_j = \dim_\kk(H^i_\m(M)_j)$. By Proposition \ref{proposition LC}, for all $0 \leq i \leq n-3$ 
we have graded $\kk$-vector space isomorphisms $H^i_\m(M) \cong \MM_i \otimes_\kk H^i_{\m}(S_i) \cong \bigoplus_{j \in \ZZ} H^i_\m(S_i(-j))^{\oplus r_{i,j}}$, where $r_{i,j} = \dim_\kk\{m \in \MM_i \mid \deg(m) = j\}$. Thus
\begin{equation} \label{Eq senza N2}
\sum_{i=0}^{n-3} \sum_{j \in \ZZ} h^i_j u^iz^j = \sum_{i=0}^{n-3} \sum_{j \in \ZZ} r_{i,j} \HS(H^\bullet_\m(S_i(-j))).
\end{equation}
On the other hand, for $n-2 \leq i \leq n$ we have $H^i_\m(M) \cong H^{i-(n-2)}_{\m_2}(\NNN_2) \otimes_\kk H^{n-2}_{\m_{n-2}}(S_{n-2})$. Therefore by (\ref{Eq Delta}) we have 
\[
\sum_{i=n-2}^n \sum_{j \in \ZZ} h^i_j u^iz^j = \frac{u^{n-2}}{(z-1)^{n-2}} \HS(H^\bullet_{\m_2}(\NNN_2)).
\]
If $\NNN_2 = 0$, then the proof is complete, since (\ref{Eq senza N2}) gives the desired decomposition. If $\NNN_2 \ne 0$, then it is a finitely generated $\ZZ$-graded $S_2$-module. It follows from \cite[Theorem 4.6]{DSS} that 
\begin{align*}
\HS(H^\bullet_{\m_2}(\NNN_2)) & = \sum_{i=0}^2 \sum_{j \in \ZZ} s_{i,j} \HS(H^\bullet_{\m_2}(S_i(-j))) + \sum_{m>0} \sum_{j \in \ZZ} s'_{m,j} \HS(H^\bullet_{\m_2}(\m_2^m(-j))),
\end{align*}
where $s_{i,j} \in \ZZ_{\geq 0}$ 
and $s_{m,j}' \in \QQ_{\geq 0}$ are all but finitely many equal to zero.

Applying (\ref{Eq Delta}) to the module $J^m \cong \m_2^m \otimes_\kk \kk[x_3,\ldots,x_n]$, we obtain the following relation:
\[
\HS(H^\bullet_\m(J^m)) = \frac{u^{n-2}}{(z-1)^{n-2}}\HS(H^\bullet_{\m_2}(\m_2^m)).
\]
Similarly, for $i=0,1,2$ we get $\displaystyle \HS(H^\bullet_\m(S_{i+(n-2)})) = \frac{u^{n-2}}{(z-1)^{n-2}}\HS(H^\bullet_{\m_2}(S_i)) $. 

Putting these relations together 
we see that:
\begin{equation}
\label{Eq N2}
\sum_{i=n-2}^n \sum_{j \in \ZZ} h^i_j u^iz^j = \sum_{i=n-2}^n \sum_{j \in \ZZ} r_{i,j} \HS(H^\bullet_{\m}(S_i(-j))) + \sum_{m>0} \sum_{j \in \ZZ} r'_{m,j} \HS(H^\bullet_{\m}(J^m(-j))),
\end{equation}
where $r_{i,j} = s_{i-(n-2),j}$ and $r'_{m,j} = s'_{m,j}$.
Finally, summing (\ref{Eq senza N2}) and (\ref{Eq N2}) up and passing to local cohomology tables gives the desired decomposition:
\[
[H^\bullet_\m(M)] = \sum_{i=0}^n \sum_{j \in \ZZ} r_{i,j} [H^\bullet_\m(S_i(-j))] + \sum_{m>0} \sum_{j \in \ZZ}r_{m,j}' [H^\bullet_\m(J^m(-j))].
\]
For minimality, the strategy of the proof is completely analogous to that of Theorem \cite[Theorem 4.6]{DSS}, combined with the use of (\ref{Eq Delta}) as above.
\end{proof} 
As a corollary of the proof, we obtain a very special decomposition for local cohomology tables in the sequentially Cohen-Macaulay case.
\begin{corollary} \label{Corollary LC decomposition SCM}
Let $S=\kk[x_1,\ldots,x_n]$ with the standard grading, and $M$ be a finitely generated $\ZZ$-graded $S$-module of dimension $d$. Assume that $M$ is sequentially Cohen-Macaulay. Then
\[
[H^\bullet_\m(M)] = \sum_{i=0}^d \sum_{j \in \ZZ} r_{i,j} [H^\bullet_\m(S_i(-j))],
\]
where $r_{i,j} \in \ZZ_{\geq 0}$, and all but finitely many of them are zero.
\end{corollary}
\begin{proof}
By Remark \ref{Remark NN}, we may assume that $n=d$. By Corollary \ref{gin SCM}  (or \cite[Theorem 3.1]{HerzogSbarra}), we can replace $M$ by its general initial module, and assume that $M$ is $\ZZ \times \ZZ^{d-1}$-graded, and that $x_d, x_{d-1},\ldots,x_1$ forms a filter regular sequence on $M$. By Proposition \ref{proposition LC}, for all $i=0,\ldots,d$ we have graded $\kk$-vector spaces $\MM_i$ and graded isomorphisms $H^i_\m(M) \cong 
\MM_i \otimes_\kk H^i_{\m}(S_i) \cong \bigoplus_{j \in \ZZ} H^i_{\m}(S_i(-j))^{\oplus r_{i,j}}$, 
where
$r_{i,j} = \dim_\kk\{m \in \MM_i \mid \deg(m)=j\}$, and the proof is complete.
\end{proof}

\subsection{Supporting hyperplanes} \label{Subsection hyperplanes}

The goal is to provide a description of the supporting hyperplanes of the cone of local cohomology tables of modules with sufficiently large $\Edepth$. As done in \cite{DSS}, to do so we must reduce to ``finitely supported'' local cohomology tables. We will now explain this process.

We start by recalling a notation introduced in \cite[Notation 4.1]{DSS}.

\begin{notation}
Let $N = \bigoplus_{j\in \ZZ}N_j$ 
be a $\ZZ$-graded $\kk$-vector space that satisfies $\dim_\kk(N_j) < \infty$ for all $j \in \ZZ$. For $t \geq 0$ we define a row vector $\Delta^t(N)$ inductively as follows. If $t=0$ then we set $\Delta^0(N) = \HF(N)$. If $t>0$, we define $\Delta^t(N)$ to be the vector whose $j$-th entry is $\Delta^t(N)_{j} = \Delta^{t-1}(N)_{j} - \Delta^{t-1}(N)_{j+1}$.
\end{notation}

Let $S=\kk[x_1,\ldots,x_n]$, with the standard grading. Recall that the local cohomology table of a finitely generated $\ZZ$-graded $S$-module is a point in the $\QQ$-vector space of $(n+1) \times \mathbb Z$-matrices ${\rm Mat}_{n+1,\ZZ}(\QQ)$. We now define an operator $\Delta(-)$, which takes a local cohomology table
$[H^\bullet_\m(M)]$ and transforms it into a new table $\Delta[H^\bullet_\m(M)]$ as follows:
\[
[H^\bullet_\m(M)] = \begin{bmatrix} \HF(H^0_\m(M)) \\ \HF(H^1_\m(M)) \\ \vdots \\  \HF(H^n_\m(M))\end{bmatrix} \ \mapsto \ \Delta[H^\bullet_\m(M)] = 
\begin{bmatrix} \Delta^0(H^0_\m(M)) \\ \Delta^1(H^1_\m(M)) \\ \vdots \\  \Delta^n(H^n_\m(M))\end{bmatrix}.
\]

It is important to observe that, even though a local cohomology table typically has infinitely many non-zero entries, the table $\Delta[H^\bullet_\m(M)]$ only has finitely many non-zero entries. This follows from the fact that the functions $\HF(H^i_\m(M)): j \mapsto \dim_\kk(H^i_\m(M)_j) = \dim_\kk(\Ext^{n-i}_S(M,S(-n))_{-j})$ coincide with polynomials of degree at most $i-1$ for $j \ll 0$.

From the point of view of Hilbert series, observe that this process corresponds to the transformation:
\[
\HS(H^\bullet_\m(M))=\sum_{i=0}^n\sum_{j \in \ZZ} h^i_j u^iz^j \mapsto \DHS(H^\bullet_\m(M)) = \sum_{i=0}^n \sum_{j \in \ZZ} h^i_j u^i z^j(1-z^{-1})^i,
\]
and from the fact that $\Delta[H^\bullet_\m(M)]$ has finite support we deduce that $\DHS(H^\bullet_\m(M))$ is a Laurent polynomial in $\ZZ[z^{\pm 1},u]$. Since local cohomology modules are zero in sufficiently high degrees, it can be checked that, given two modules $M$ and $N$, $\HS(H^\bullet_\m(M)) = \HS(H^\bullet_\m(N))$ if and only if $\DHS(H^\bullet_\m(M)) = \DHS(H^\bullet_\m(N))$. In the rest of the section, we will often switch from the point of view of tables to that of Hilbert series, and vice versa.

For convenience, we will adopt the convention that the first row of a matrix $A=(a_{i,j})  \in {\rm Mat}_{n+1,\ZZ}(\QQ)$ corresponds to the index $i=0$, the second row to the index $i=1$, and so on. This makes the association between the rows of $A$ and invariants of local cohomology modules with a given cohomological index more natural.


Let $\mathbb M$ be the $\QQ$-vector subspace of ${\rm Mat}_{n+1,\ZZ}(\QQ)$ generated by tables with finitely many non-zero entries. Then $\mathbb M$ can be filtered as $\mathbb M = \bigcup_{a<b} \mathbb M_{[a,b]}$, where
\[
\mathbb M_{[a,b]} := \left\{A  \in \mathbb M \mid a_{i,j} = 0 \text{ whenever } j < a, j > b \right\}.
\]
Let $\mathcal C_{[a,b]}$ be the cone spanned by tables $\Delta[H^\bullet_\m(M)]$ that are contained in $\mathbb M_{[a,b]}$, where $M$ is a $\ZZ$-graded $S$-module such that $\Edepth(M) \geq n-2$. Let $\mathcal C^{\rm seq}_{[a,b]}$ be the cone spanned by tables $\Delta[H^\bullet_\m(M)]$ that are contained in $\mathbb M_{[a,b]}$, where $M$ is a sequentially Cohen-Macaulay $\ZZ$-graded $S$-module. Clearly, $\mathcal C^{\rm seq}_{[a,b]} \subseteq \mathcal C_{[a,b]}$. We start with a description of the smaller cone.

\begin{proposition} \label{hyperplanes SCM}
For integers $a<b$, consider the cone defined by $\mathcal D^{{\rm seq}}_{[a,b]} = \{A \in \mathbb{M}_{[a,b]} \mid a_{i,j} \geq 0 \text{ for all } 0 \leq i \leq n \text{ and } j \in \ZZ\}$. Then $\mathcal{C}^{{\rm seq}}_{[a,b]} = \mathcal D^{{\rm seq}}_{[a,b]}$.
\end{proposition}
\begin{proof}
Let $H \in \mathcal{C}^{{\rm seq}}_{[a,b]}$ be a table, that we may assume being equal to $\Delta[H^\bullet_\m(M)]$ for some sequentially Cohen-Macaulay graded $S$-module $M$. 
By Corollary \ref{Corollary LC decomposition SCM} we have that $[H^\bullet_\m(M)]$ can be written as a sum, with positive coefficients and shifts, of tables of the form $[H^\bullet_\m(S_i)]$. It is easy to see that, for all $i$, the table $\Delta[H^\bullet_\m(S_i)]$ has non negative entries, and thus $H \in \mathcal D^{\rm seq}_{[a,b]}$.

Conversely, let $H\in \mathcal D^{\rm seq}_{[a,b]}$. We can represent $H$ as a Laurent polynomial, namely $P(H)=\sum_{i=0}^n \sum_{j \in \ZZ} r_{i,j} u^i z^j$, where $r_{i,j}$ denotes the entry of $H$ in position $(i,j)$. Observe that $r_{i,j}\geq 0$ by assumption. By (\ref{Eq Delta}) we have that $\HS(H^\bullet_\m(S_i(-i)))  = z^i\HS(H^\bullet_\m(S_i)) = \frac{(uz)^i}{(z-1)^i}\HS(H^\bullet_\m(\kk)) = \frac{u^i}{(1-z^{-1})^i}$, and thus $\DHS(H^\bullet_\m(S_i(-i))) = u^i$. Therefore
\[
P(H) = \sum_{i=0}^n \sum_{j \in \ZZ} r_{i,j} u^iz^j =\sum_{i=0}^n\sum_{j \in \ZZ} r_{i,j} z^j\DHS(H^\bullet_\m(S_i(-i))) = \sum_{i=0}^n\sum_{j \in \ZZ} r_{i,j} \DHS(H^\bullet_\m(S_i(-j-i))).
\]
In the language of tables, this means that $H = \sum_{i=0}^n \sum_{j\in \ZZ} r_{i,j} \Delta[H^\bullet_\m(S_i(-j-i))]$. Since each module $S_i(-j-i)$ is sequentially Cohen-Macaulay, it follows that $\Delta[H^\bullet_\m(S_i(-j-i))] \in \mathcal C^{\rm seq}_{[a,b]}$ whenever $r_{i,j} \ne 0$. Thus, $H \in \mathcal C^{\rm seq}_{[a,b]}$.
\end{proof}

To describe the cone $\mathcal C_{[a,b]}$ we need to introduce more functionals, which come from \cite{DSS}.
\begin{definition} \label{defn hyperplanes} Let $A = (a_{i,j}) \in \mathbb M$. For $j \in \ZZ$ we set 
\[
\tau_j(A)= a_{n-1,j} + \sum_{s \leq j-1} a_{n,s}.
\]
Given an integer $m \geq 0$ and $j \in \ZZ$, we set
\[
\ds \pi_{m,j}(A)= \sum_{s > j+m} a_{n-1,s} + (m+1)a_{n-1,j+m} + \sum_{s=0}^{m-1} (s+1)a_{n,j+s}.
\] 
Finally, given an integer $0 \leq i \leq n$ and $j \in \ZZ$, we let $\mu_j^{(i)} = a_{i,j}$.
\end{definition}

Let $a<b$ be integers, and consider the following list of functionals on $\mathbb{M}$: 
\[
\ds \mathcal H_{[a,b]}= \left\{ \begin{array}{ll}
\mu^{(i)}_j & \text{ for } a \leq j \leq b, 0 < i \leq n-2 \text{ and } i=0,n,\\ \\
\tau_{j}   & \text{ for } a \leq j < b, \\ \\
\pi_{0,j} & \text{ for } a+1 \leq j \leq b, \\ \\
\pi_{m,j} & \text{ for } a+1 \leq m \leq b - 2,  a+1 \leq j < b - m
\end{array}\right\}.
\]

\begin{remark} \label{Remark 2-dim hyperplanes} Observe that the functionals $\tau_j, \pi_{0,j}, \pi_{m,j}$ are precisely those appearing in \cite[Theorem 6.2]{DSS}, with the only difference that the variables $a_{1,j}$ are here replaced by $a_{n-1,j}$ and the variables $a_{2,j}$ are replaced by $a_{n,j}$. Roughly speaking, if a matrix $A \in \mathbb M_{[a,b]}$ satisfies $\tau_j(A) \geq 0, \pi_{0,j}(A) \geq 0, \pi_{m,j}(A) \geq 0$ and $\mu_j^{(i)}(A) \geq 0$ for $i=n-2, n$, then the submatrix corresponding to the last three rows of $A$ satisfies the inequalities given by the functionals from \cite[Theorem 6.2]{DSS}.
\end{remark}

\begin{theorem}\label{facet thm} For integers $a<b$, consider the cone $\mathcal D_{[a,b]} = \{A \in \mathbb M_{[a,b]} \mid \phi(A) \geq 0$ for all $\phi \in \mathcal{H}_{[a,b]}\}$. We have $\mathcal D_{[a,b]} = \mathcal C_{[a,b]}$.
\end{theorem}
\begin{proof}
If $n \leq 1$, then every finitely generated $S$-module is sequentially Cohen-Macaulay, and the result trivially follows from Proposition \ref{hyperplanes SCM}. If $n=2$, the result follows from \cite[Theorem 6.2]{DSS}. Henceforth, we will assume that $n \geq 3$.

Let $H \in \mathcal{C}_{[a,b]}$ be a table. As in Proposition \ref{hyperplanes SCM}, we may assume that $H=\Delta[H^\bullet_\m(M)]$, where $M$ is a graded $S$-module of Krull dimension $d \leq n$, and $\Edepth(M) \geq n-2$. 

Observe that, if $d<n$, then $\Edepth(M) \geq n-2 \geq d-1$, and thus $M$ is sequentially Cohen-Macaulay by Proposition \ref{Proposition SCM and Edepth}. It follows from Proposition \ref{hyperplanes SCM} that $H$ satisfies $\mu_j^{(i)}(H) \geq 0$ for all $a \leq j \leq b$ and $0 \leq i \leq n$. From this, it can easily be checked that $\phi(H) \geq 0$ for all $\phi \in \mathcal H_{[a,b]}$, that is, $H \in \mathcal D_{[a,b]}$.

Now assume that $d=n$. 
Apply the operator $\Delta(-)$ to a decomposition of $[H^\bullet_\m(M)]$ obtained from Theorem \ref{theorem decomposition LC Edepth}:
\[
\Delta[H^\bullet_\m(M)] =  \sum_{i=0}^n \sum_{j \in \ZZ} r_{i,j}\Delta[H^\bullet_{\m}(S_i(-j))] + \sum_{m > 0} \sum_{j \in \ZZ} r'_{j,m} \Delta[H^{\bullet}_\m(J^m(-j))],
\]
where $J=(x_{1},x_2)S$, the rational numbers $r_{i,j}, r_{j,m}'$ are non-negative, and only finitely many of them are not zero. Thus, it suffices to show that every single table appearing on the right-hand side of the equation satisfies the inequalities defining $\mathcal D_{[a,b]}$. For tables of the form $\Delta[H^\bullet_\m(S_i(-j))]$, the argument is the same as the one used above when $M$ is sequentially Cohen-Macaulay. For $\Delta[H^{\bullet}_\m(J^m(-j))]$, we pass through its Hilbert series $\DHS(H^\bullet_\m(J^m(-j)))$. Using (\ref{Eq Delta}), straightforward calculations give that 
\[
\DHS(H^\bullet_\m(J^m(-j))) = (uz^{-1})^{n-2} \DHS(H^\bullet_{\m_2}(\m_2^m(-j))) =  u^{n-2} \DHS(H^\bullet_{\m_2}(\m_2^m(-j+n-2))).
\] 
We conclude that the table $\Delta[H^{\bullet}_\m(J^m(-j))]$ satisfies the desired inequalities because the table $\Delta[H^\bullet_{\m_2} (\m_2^m(-j+n-2))]$ of the $S_2$-module $\m_2^m(-j+n-2)$ satisfies analogous functionals in dimension two, by \cite[Theorem 6.2]{DSS} (see Remark \ref{Remark 2-dim hyperplanes}).

Conversely, let $H \in \mathcal D_{[a,b]}$. We show that $H$ belongs to $\mathcal C_{[a,b]}$ by constructing the first $n-2$ rows as a first approximation, and then adding the last three. Let $r_{i,j}$ be the $(i,j)$-th entry of $H$, and consider the associated Laurent polynomial $P(H) = \sum_{i=0}^n\sum_{j \in \ZZ} r_{i,j} u^iz^j$. Since $\mu^{(i)}_j(H) \geq 0$ for all $0 \leq i \leq n-3$, as in Proposition \ref{hyperplanes SCM} we have that
\begin{equation}
\label{Eq prime righe}
\sum_{i=0}^{n-3}\sum_{j \in \ZZ} r_{i,j} u^iz^j = \sum_{i=0}^{n-3} \sum_{j \in \ZZ} r_{i,j} \DHS(H^\bullet_\m(S_i(-j-i))),
\end{equation}
with $r_{i,j} \geq 0$. As already observed in Proposition \ref{hyperplanes SCM}, we have that $\Delta[H^\bullet_\m(S_i(-j-i))] \in \mathcal C^{\rm seq}_{[a,b]} \subseteq \mathcal C_{[a,b]}$ whenever $r_{i,j} \ne 0$.
We now construct the last three rows of $H$. 
By Remark \ref{Remark 2-dim hyperplanes}, the submatrix $H' \in {\rm Mat}_{3,\ZZ}(\QQ)$ consisting of the last three rows of $H$ satisfies the inequalities given by the functionals from \cite[Theorem 6.2]{DSS}. As a consequence of such theorem, there exist 
finitely generated $\ZZ$-graded $S_2$-modules $N_1,\ldots,N_p$ and positive rational numbers $t_1,\ldots,t_p$ such that $H' = \sum_{s=1}^p t_s \Delta[H^\bullet_{\m_2}(N_s)]$. Let $P(H') = \sum_{i=0}^2 r_{i+n-2,j} u^iz^j$ denote the Laurent polynomial associated to $H'$. Passing to Hilbert series, the equation above gives that $P(H') = \sum_{s=1}^p t_s \DHS(H^\bullet_{\m_2}(N_s))$. 
For $s=1,\ldots p$, we let  $N_s' = \left(N_s \otimes_\kk \kk[x_3,\ldots,x_{n}]\right)(2-n)$, so that by (\ref{Eq Delta}) and with a straightforward calculation we get
\begin{equation}
\label{Eq ultime righe}
\sum_{s=1}^p t_s \DHS(H^\bullet_\m(N_s')) = u^{n-2} \sum_{s=1}^p t_s  \DHS(H^\bullet_{\m_2}(N_s)) = u^{n-2}P(H') = \sum_{i=n-2}^n \sum_{j \in \ZZ} r_{i,j}u^iz^j.
\end{equation}
Observe that $\Edepth(N_s') \geq n-2$, so that $\Delta[H^\bullet_\m(N'_s)] \in \mathcal C_{[a,b]}$ for all $s$. 
Putting (\ref{Eq prime righe}) and (\ref{Eq ultime righe}) together, from the point of view of tables we finally get
\[
H=  \sum_{i=0}^{n-3} \sum_{j \in \ZZ} r_{i,j}\Delta[H^\bullet_\m(S_i(-j-i))] + \sum_{s=1}^p t_s \Delta[H^\bullet_\m(N_s')] \in \mathcal C_{[a,b]}. \qedhere
\]
\end{proof}

\section{If the socle of local cohomology modules fits}
\label{Section soclelemma}
Let $S=\kk[x_1,\ldots,x_n]$ with the standard grading, and $\m=(x_1,\ldots,x_n)$. Let $M$ be a $\ZZ$-graded $S$-module. We denote by $\soc(M) = \ann_M(\m)$ the socle of $M$. In \cite{KustinUlrich}, Kustin and Ulrich prove that if $I \subseteq J$ are two homogeneous ideals such that $S/I$ is Artinian, and that satisfy the pointwise inequality $\HF(\soc(S/I)) \leq \HF(\soc(S/J))$, then $I=J$. The proof can easily be adapted to submodules $A \subseteq B$ of a fixed $S$-module $C$ such that $C/A$ is Artinian, and we include it here for the sake of completeness. 

\begin{lemma} \label{lemmainequalitysocles Artinian} Let $S=\kk[x_1,\ldots,x_n]$, and $C$ be a finitely generated graded $S$-module. Let $A \subseteq B$ be two graded submodules of $C$ such that $C/A$ is Artinian. If $\HF(\soc(C/A)) \leq \HF(\soc(C/B))$, then $A=B$.
\end{lemma}
\begin{proof}
Assume that the natural map $\soc(C/A) \to \soc(C/B)$ is not injective. Let $j$ be the largest degree in which it is not injective. Then, because of our assumptions, the map $\soc(C/A)_j \to \soc(C/B)_j$ is not surjective either. This means that there exists a non-zero $\overline{c} \in \soc(C/B)_j$ such that its lift $c \in C$ satisfies $\m c \not\subseteq A$. Choose an element $f \in S$ of largest positive degree with the property that $fc \notin A$. Then $\overline{fc} \in \bigoplus_{i > j} \soc(C/A)_{i}$  
by construction, and such an element maps to zero in $\soc(C/B)$. This contradicts our choice of $j$, and concludes the proof that the map $\soc(C/A) \to \soc(C/B)$ is injective. To conclude the proof, observe that there is an exact sequence $0 \to \soc(B/A) \to \soc(C/A) \to \soc(C/B)$. As we proved that the rightmost map is injective, we have that $\soc(B/A)=0$. However, $B/A$ has finite length, so its socle cannot be zero unless $A=B$.
\end{proof}

The following is a generalization to Kustin and Ulrich's socle lemma to the non-Artinian case.
\begin{theorem} \label{soclelemma}
Let $S=\kk[x_1,\ldots,x_n]$, with the standard grading, and $F$ be a graded free $S$-module. Let $A \subseteq B$ be two graded submodules of $F$. Let $\ell_1,\ldots,\ell_t$ be a filter regular sequence for both $F/A$ and $F/B$ consisting of linear forms such that $\HF((A+(\ell_1,\ldots,\ell_t)F)^{\sat}) = \HF((B+(\ell_1,\ldots,\ell_t)F)^{\sat})$.
Furthermore, assume that $\HF(\soc(H^i_\m(F/A))) \leq \HF(\soc(H^i_\m(F/B)))$ for all $0 \leq i \leq t$. If $\min\{\Edepth(F/A),\Edepth(F/B)\} \geq t-1$, then $A=B$.
\end{theorem}
\begin{proof}
We prove the result by induction on $t \geq 0$, treating the cases $t=0$ and $t=1$ separately. If $t=0$ we have $A^{\sat} = B^{\sat}$ by assumption. Call such module $C$. Moreover, observe that $H^0_\m(F/A) = A^{\sat}/A = C/A$ and $H^0_\m(F/B) = B^{\sat}/B = C/B$. By assumption, we have that $\HF(\soc(C/A)) \leq \HF(\soc(C/B))$, hence $A=B$ follows from Lemma \ref{lemmainequalitysocles Artinian}.

Now assume that $t=1$, and set $\ell=\ell_1$. We observe that $H^i_\m(F/A) \cong H^i_\m(F/A^{\sat})$ for all $i >0$. We have a graded short exact sequence
\[
\xymatrix{
0 \ar[r] & F/A^{\sat}(-1) \ar[r]^-{\cdot \ell} & F/A^{\sat} \ar[r] & F/(A^{\sat}+\ell F) \ar[r] & 0,
}
\]
which induces a graded long exact sequence on local cohomology
\[
\xymatrix{
0 \ar[r] & H^0_\m(F/(A^{\sat}+\ell F)) \ar[r] & H^1_\m(F/A)(-1) \ar[r]^-{\cdot \ell} & H^1_\m(F/A) \ar[r] & \cdots
}
\]
\[
\xymatrix{
\cdots \ar[r] & H^{i-1}_\m(F/(A^{\sat}+\ell F)) \ar[r] & H^i_\m(F/A)(-1) \ar[r]^-{\cdot \ell} & H^i_\m(F/A) \ar[r] & \cdots
}
\]
Taking $\Hom_S(\kk,-)$ in the first part of the sequence gives
\[
\xymatrix{
0 \ar[r] & \soc(H^0_\m(F/(A^{\sat}+\ell F))) \ar[r] & \soc(H^1_\m(F/A))(-1) \ar[r]^-{\cdot \ell} & \soc(H^1_\m(F/A)).
}
\]
Multiplication by $\ell$ is zero on socles, hence $\soc(H^0_\m(F/(A^{\sat}+\ell F))) \cong \soc(H^1_\m(F/A))(-1)$. An analogous argument holds for $F/B$. Hence we still have an inequality $\HF(\soc(H^0_\m(F/(A^{\sat}+\ell F)))) \leq \HF(\soc(H^0_\m(F/(B^{\sat}+\ell F))))$. Let $C=(A+\ell F)^{\sat}$, which is also equal to $(B+\ell F)^{\sat}$ by assumption. Finally, observe that $H^0_\m(F/(A^{\sat}+\ell F)) = (A^{\sat}+\ell F)^{\sat}/(A^{\sat}+\ell F) = C/(A^{\sat}+\ell F)$. Similarly, $H^0_\m(F/(B^{\sat}+\ell F)) = C/(B^{\sat}+\ell F)$. Since we have 
containments $(A^{\sat} + \ell F) \subseteq (B^{\sat}+\ell F) \subseteq C$, with an appropriate inequality on socles, it follows from Lemma \ref{lemmainequalitysocles Artinian} that $A^{\sat}+\ell F= B^{\sat} + \ell F$. Since we always have a containment $A^{\sat} \subseteq B^{\sat}$, and $\ell$ is a non-zero divisor on $F/A^{\sat}$ and $F/B^{\sat}$, we must have $A^{\sat}=B^{\sat}$. To conclude, we now argue as in the case $t=0$.

We now assume that $t \geq 2$. Set $\ell=\ell_1$. Since $\min\{\Edepth(F/A),\Edepth(F/B)\} \geq t-1 >0$ the long exact sequence in local cohomology breaks into short exact sequences
\[
\xymatrix{
0 \ar[r] & H^i_\m(F/(A^{\sat}+\ell F)) \ar[r] & H^{i+1}_\m(F/A)(-1) \ar[r]^{\cdot \ell } & H^{i+1}_\m(F/A) \ar[r] & 0
}
\]
for all $i \geq 0$. In particular, taking $\Hom_S(\kk,-)$ and using again that multiplication by $\ell$ is zero on socles, we get isomorphisms $\soc(H^i_\m(F/(A^{\sat}+\ell F))) \cong \soc(H^{i+1}_\m(F/A)(-1))$ for all $i \geq 0$. A similar argument holds for $F/(B^{\sat}+\ell F)$, so that $\HF(\soc(H^i_\m(F/(A^{\sat}+\ell F)))) \leq \HF(\soc(H^i_\m(F/(B^{\sat}+\ell F))))$ for all $0 \leq i \leq t-1$. 

One can check that $\ell_2,\ldots,\ell_t$ are still a filter regular 
sequence for both $F/(A^{\sat}+\ell F)$ and $F/(B^{\sat}+\ell F)$. Moreover,
\begin{align*}
(A^{\sat}+\ell F+(\ell_2,\ldots,\ell_t)F)^{\sat} &= (A+(\ell_1,\ldots,\ell_t) F)^{\sat} \\
&=  (B+(\ell_1,\ldots,\ell_t)F)^{\sat} \\
& = (B^{\sat}+\ell F +(\ell_2,\ldots,\ell_t)F)^{\sat},
\end{align*}
and thus $\HF((A^{\sat} + \ell F + (\ell_2 ,\ldots,\ell_t) F)^{\sat}) = \HF((B^{\sat} + \ell F + (\ell_2 ,\ldots,\ell_t) F)^{\sat})$.

Since $\min\{\Edepth(F/(A^{\sat}+\ell F)),\Edepth(F/(B^{\sat}+\ell F))\} \geq t-2$ by Proposition \ref{proposition Edepth modulo filter regular element}, it follows from our inductive hypothesis that $A^{\sat}+\ell F=B^{\sat} + \ell F$. To conclude the proof we now proceed as in the case $t=1$.
\end{proof}

If both $F/A$ and $F/B$ are sequentially Cohen-Macaulay, then we only need to check the condition on the Hilbert functions of the socles of (all) local cohomology modules.  
\begin{corollary} \label{soclelemma sequentially} Let $S=\kk[x_1,\ldots,x_n]$, with the standard grading, and $F$ be a graded free $S$-module. Let $A \subseteq B$ be two graded submodules of $F$ that satisfy $\HF(\soc(H^i_\m(F/A))) \leq \HF(\soc(H^i_\m(F/B)))$ for all $i\in \ZZ$. If $F/A$ and $F/B$ are sequentially Cohen-Macaulay, then $A=B$.
\end{corollary}
\begin{proof}
Observe that the inequality between socles forces 
$F/A$ and  $F/B$ to have the same Krull dimension $d$. 
Let $\ell_1,\ldots,\ell_d$ be a filter regular sequence for both $F/A$ and $F/B$. Since $\ell_1,\ldots,\ell_d$ is a full system of parameters for $F/A$, we have that $(A+(\ell_1,\ldots,\ell_d)F)^{\sat} = F$, because $F/(A+(\ell_1,\ldots,\ell_d)F)$ has finite length. A similar statement holds for $B$ as well. The Corollary now follows from Theorem \ref{soclelemma}, since the assumptions that $\HF((A+(\ell_1,\ldots,\ell_d)F)^{\sat}) = \HF((B+(\ell_1,\ldots,\ell_d)F)^{\sat})$ and $\min\{\Edepth(F/A),\Edepth(F/B)\} \geq d-1$ are trivially satisfied.
\end{proof}

\bibliographystyle{alpha}
\bibliography{References}
\end{document}